\documentclass[oneside,a4paper,english]{smfart}
\usepackage[textheight=21cm, textwidth=13cm]{geometry}  
\usepackage[english]{babel}
\author[H. Rajabzadeh]{Hesam Rajabzadeh}
\address{Institute for Research in Fundamental Sciences (IPM), Tehran, Iran.}
\email{rajabzadeh@ipm.ir, rajabzadeh.hesam@gmail.com}

\author[P. Safaee]{Pedram Safaee}
\address{CMLS, École polytechnique, 91128 Palaiseau, France.}
\email{pedram.safaee@polytechnique.fr, pedram.safaee95@gmail.com}

\title[Twisted cocycle for IETs: Invariant structures and Lyapunov spectrum]{Twisted cocycle for interval exchange transformations: Invariant structures and Lyapunov spectrum}
\date{28 January 2025}

\usepackage[T1]{fontenc}
\usepackage[utf8]{inputenc}
\usepackage{xspace}

\usepackage{amsmath, amsthm, amsfonts, amssymb, smfthm}
\usepackage{authblk}
\usepackage{cases}
\usepackage{hyperref}
\usepackage{esvect}
\usepackage{bbm}
\usepackage{accents}
\usepackage{graphicx}
\usepackage{caption}
\usepackage{MnSymbol}
\usepackage{relsize}
\usepackage{url}
\usepackage{array}

\usepackage{tikz-cd}

\usepackage{nicematrix}

\newtheorem{thm}{Theorem}[section]
\newtheorem*{thm*}{Theorem}

\newtheorem*{theoremain}{Main Theorem}
\newtheorem*{corollary}{Corollary}

\newtheorem{lemma}[thm]{Lemma}
\newtheorem{cor}[thm]{Corollary}
\newtheorem{rmk}[thm]{Remark}

\newtheorem{definition}[thm]{Definition}
\numberwithin{equation}{section}

\newtheorem{example}{Example}

\usepackage{xcolor}
 \usepackage{hyperref}
          \hypersetup{colorlinks, citecolor=blue, filecolor=blue, linkcolor=black, urlcolor=blue}
     \usepackage{soul}
     \usepackage[numbers]{natbib}

\newcommand{\prealpha}{\alpha^{-1}}

\newcommand{\tw}{\mathcal{B}}

\newcommand{\n}{\vec{\textbf{n}}}

\newcommand{\one}{{\textbf{1}}}
\newcommand{\real}{\boldsymbol{R}}

\newcommand{\vs}{v_{\pi,\eta}^S}

\newcommand{\toral}{\mathcal{R}_{\mathbb{T}^d}}

\newcommand{\w}[1]{w_{\eta}(#1)}

\begin{document}

\maketitle
\makeatother

\begin{abstract}
This paper investigates the algebraic and dynamical properties of the twisted cocycle, a \(\mathrm{GL}(d, \mathbb{C}\))-valued cocycle defined over the toral extension of the Zorich (Rauzy-Veech) renormalization for interval exchange transformations (IET). As a natural generalization of the Zorich cocycle, the twisted cocycle plays a central role in studying the asymptotic growth of twisted Birkhoff sums which in turn provide a suitable tool for obtaining fine spectral information about IETs and translation flows such as the local dimension of spectral measures and quantitative weak mixing. Although it shares similarities with the classical (untwisted) Zorich cocycle, structural differences make its analysis more challenging. Our results yield a block-form decomposition into invariant and covariant subbundles allowing us to demonstrate the existence of $\kappa+1$ zero exponents with respect to a large class of natural invariant measures where $\kappa$ is an explicit integer depending on the permutation. We establish the symmetry of the Lyapunov spectrum by showing the existence of a family of non-degenerate invariant symplectic forms. As a corollary, we prove that for rotation-type permutations, the twisted cocycle has a degenerate Lyapunov spectrum with respect to certain natural ergodic invariant measures in contrast with the higher genus case where the existence of at least one positive Lyapunov exponent is guaranteed by previous work of the authors. In the appendix, we apply our result about the invariant section to substitution systems and prove the pure singularity of the spectrum for a substantially large class of substitutions on two letters, while greatly simplifying the proof for some systems previously known to have this property. 
\end{abstract}

\author{Hesam Rajabzadeh}

\section{Introduction}

Interval exchange transformations (IET) are one-dimensional dynamical systems that naturally generalize the notion of rotations on the circle that correspond to interval exchanges on two subintervals. IETs and their generalizations (so-called GIETs) provide simple one-dimensional (rather combinatorial) models that facilitate the study of flows on surfaces. 


A key approach to understanding the dynamics of IETs and translation flows is to analyze the associated renormalization dynamics and linear cocycles. Although recurrence of the Teichmuller geodesic flow orbits to compact sets implies unique ergodicity \cite{MasurIETMeasuredFoliations, VeechGaussmeasure}, the study of the Kontsevich-Zorich (KZ) cocycle over the $\mathrm{SL}(2, \mathbb{R})$-action on the moduli space of abelian differentials plays a fundamental role in establishing quantitative statements regarding the rate of ergodicity (see, for example, \cite{ ForniDeviation,AvilaVianaSimplicity,AthreyaForni}). In this direction, many researchers have worked on relating the Lyapunov spectrum of the KZ cocycle or its discrete-time analog called the Zorich cocycle to the ergodic properties of individual IETs and translation flows (see for example \cite{Zorich97deviation, MMYRothtype, AvilaForniweak}). These are simply instances of the utility of the more general philosophy that there is a dictionary between the dynamical properties of translation flows/IETs and the properties of the associated renormalization dynamics \cite{Katoknomixing, Chaika-Annals, Chaika-Hubert, Skripchenko-Troubetzkoy, Chaika-Eskin-Mobius, Kanigowski-Lemanczyk-Radziwill, ChaikaEskin-self-joining}.


Weak mixing for typical IETs or in typical directions for translation flows has been well studied \cite{KatokStepinWeak,VeechMetricTheory,   NogueiraRudolphTopweak,  Sinai-Ulcigrai,Bufetov-Sinai-Ulcigrai,AvilaForniweak, Ulcigrai-JMD, Ulcigrai-Annals, AvilaDelecroixweak, AvilaLeguil, WeakmixingBilliard}. In
the past 15 years, there has been an increasing interest in the problem of quantitative weak mixing (\cite{BufetovSolomyakHoldergenustwo, BufetovSolomyak-HolderRegularity, ForniTwisted, AvilaForniSafaee, Bufetov-Marshal-Solomyak_Quant.Weakmixing}). A notable development made by Bufetov and Solomyak \cite{BufetovSolomyakHoldergenustwo} in this 
field is the introduction of the twisted cocycle, a relatively new linear cocycle that has recently gained attention. Although this cocycle is defined over the toral (Rauzy-Veech) Zorich cocycle, Forni \cite{ForniTwisted} defined a similar cocycle over the Teichmuller geodesic flow. These cocycles were originally used to attack the problem of quantitative weak mixing for IETs and translation flows. To name a few other applications, we must mention that these cocycles provide an important framework for the study of other spectral problems such as determining the local dimension of spectral measures or the singularity of the spectrum (with respect to the Lebesgue measure) \cite{
Baake2,Baake,Marshall-Maldonado,Solomyak_singularity_S_adic}. In particular, Bufetov and Solomyak \cite{Bufetov-Solomyak-spectralcocycle2020} obtained a precise formula relating the local dimension of spectral measures of certain observables to the ratio of the (pointwise) top Lyapunov exponent of the twisted cocycle to the top exponent of the untwisted cocycle. A corollary of their observation is that the establishment of strict upper bounds for the top exponent of the twisted cocycle implies the singularity of the spectrum \cite{Bufetov-Solomyak-MathZ}. 
For a general introduction to the topic, see the surveys
\cite{BufetovSolomyak2023,Forni2024-survey}. 

The twisted cocycle is a $
\mathrm{GL}(d,\mathbb{C})$-valued cocycle defined over the toral extension of the Zorich (Rauzy-Veech) renormalization transformation. This cocycle bears a resemblance to the (untwisted) ordinary Zorich cocycle in that it allows for studying the asymptotics of twisted Birkhoff sums just as the Zorich cocycle
governs the rate of growth of ordinary Birkhoff sums. However, there are certain differences between the two cocycles that make the study of the twisted one more difficult. This partly explains why many of the properties of this cocycle have not yet been understood, although it has been a subject of
interest for more than a decade. For example, there is no analog of the tautological plane in the twisted realm and even the problem of positivity of the top exponent was only recently resolved in a previous work by the authors (\cite{Rajabzadeh-Safaee}). The authors used a geometric construction to show that when the genus of the associated translation surface is greater than one, the cocycle has at least one positive Lyapunov exponent with respect to natural ergodic invariant measures.

In this paper, we investigate the similarities and differences between the two cocycles and shed more light on the algebraic structure of the twisted cocycle.

\begin{theoremain}\label{thm-A}
Let $d>1$ be an integer, $\pi=(\pi_t, \pi_b)$ an irreducible permutation, and $\mathfrak{R}:=\mathfrak{R(\pi)}$ its Rauzy class. Let $g$ be the genus, and $\kappa$ the number of singularities corresponding to the Rauzy class $\mathfrak{R}$. Then the Lyapunov spectra with respect to the invariant measure $\sum_{\pi \in \mathfrak{R}} \mu_{Z} \vert_{\Delta_\pi} \times m_{\mathbb{T}^{2g}_{\pi}}$ of the twisted cocycles $\mathcal{B}$ and $(\mathcal{B}^{*})^{-1}$ are equal, symmetric with respect to zero, and contain zero with multiplicity $\kappa+1$. That is, the Lyapunov spectra take the following form
\begin{equation}
    \chi_1 \geq \cdots \geq \chi_{g-1}\geq \underbrace{0=\cdots=0}_{\kappa+1}\geq -\chi_{g-1} \geq \cdots \geq-\chi_{1}.
\end{equation}
\end{theoremain}
Although the above statement and its proof apply to all invariant measures defined in Subsection \ref{sec: Twisted Cocycle},  we chose to state the theorem as above for the sake of simplicity of the exposition. A similarity between the twisted cocycle and the ordinary Zorich cocycle is the symmetry of the spectrum. Observe, however, that there are two additional zero exponents, that is, $\kappa+1$ as opposed to $\kappa-1$ in the ordinary case. One of the main differences between the two cocycles is best illustrated in the following immediate corollary of Theorem \ref{thm-A}.
\begin{corollary} Let $d>1$ be an integer and $\pi=(\pi_t, \pi_b)$ be a rotation type permutation. Then, the Lyapunov spectrum of the twisted cocycle with respect to the ergodic invariant measure $\mu \times m_{\mathbb{T}^2}$ is degenerate, i.e., all exponents are equal to zero. 
\end{corollary}

Let us outline some of the ideas discussed in the paper. A central difficulty in our work is finding the correct family of matrices $\Omega_{\pi, \zeta}$ satisfying certain desired properties. These matrices, which may be thought of as twisted analogs of the $\Omega_{\pi}$ matrix defined by Veech \cite{VeechMetricTheory}, equip us with the right set of tools to provide a decomposition of the cocycle into block forms and construct nondegenerate invariant symplectic structures. We decompose the trivial $\mathbb{C}^d$-bundle where the twisted cocycle acts to the direct sum of three subbundles depending only on the twist factor $\zeta$ and the combinatorics $\pi$ (everywhere defined except at the zero section of $\Delta \times \mathbb{T}^d$)
\begin{equation}\label{eq:cocycle-decomposition}
    \mathbb{C}v_{\pi, \zeta} \oplus N(\pi, \zeta) \oplus \widetilde{H}(\pi, \zeta)
\end{equation}
satisfying the following  
\begin{enumerate}
    \item $N(\pi, \zeta)=Ker\Omega_{\pi, \zeta}$ is a $(\kappa-1)$-dimensional subbundle spanned by $v^{S}_{\pi, \eta}$ for $S\in \Sigma(\pi) \setminus \{S_0\}$,
    \item the dual cocycle $(\mathcal{B}^{*})^{-1}$ acts unitarily on $N(\pi, \zeta)$,
    \item $v_{\pi, \zeta}$ is an invariant section of the dual cocycle $(\mathcal{B}^{*})^{-1}$,
    \item $s_{\zeta}=\Omega_{\pi, \zeta}v_{\pi, \zeta}$ defines an invariant section (everywhere nonzero except at the origin) for the action of $\mathcal{B}$,
    \item $W_\zeta= \langle s_\zeta\rangle^{\perp}$ is a $(d-1)-$dimensional invariant subbundle for $(\mathcal{B}^{*})^{-1}$
    \item $\widetilde{H}(\pi, \zeta)= \Omega_{\pi, \zeta} W_{\zeta}$ is a $(2g-1)$-dimensional invariant subbundle under the twisted cocycle $\mathcal{B}$,
\end{enumerate}
where $S_0$ is the distinguished singularity corresponding to $0$ (see Subsection \ref{subsec: IETs and translation surfaces}). 

The direct sum in \eqref{eq:cocycle-decomposition} provides a decomposition for the action of the twisted cocycle into invariant and covariant subbundles related to the image and the kernel of the $\Omega_{\pi, \zeta}$ matrices. The decomposition and the items above show similarities and differences between the twisted and untwisted realms. While the decomposition in \eqref{eq:cocycle-decomposition} resembles that of the (untwisted) ordinary cocycle into the direct sum of the kernel and the image of $\Omega_{\pi}$, the difference is the appearance of the covariant section $v_{\pi, \zeta}$ which is not in the kernel of $\Omega_{\pi, \zeta}$. Another important difference between the twisted $\Omega_{\pi,\zeta}$ matrices and the $\Omega_\pi$ matrix is that the invariance of $\Omega_{\pi, \zeta}$'s under renormalization is only true restricted to $W_\zeta$, see Proposition \ref{inv-Omega}. Moreover, the image of $\widetilde{H}(\pi, \zeta)$ admits an invariant section $s_{\zeta}$.  As made apparent by Proposition \ref{prop:symplectic-form}, the suitable twisted analog of $H(\pi)$ is the $(2g-2)$-dimensional bundle $H(\pi, \zeta)$ obtained by quotienting $\widetilde{H}$ by $\mathbb{C}s_{\zeta}$. It turns out that the real bundle corresponding to $H(\pi, \zeta)$ will admit a family of nondegenerate invariant symplectic forms $\omega^{\mathbb{R}}_{\pi, \zeta}$ which finishes the proof of the symmetry statement in Theorem \ref{thm-A}. 

\section*{Acknowledgments} We would like to extend our gratitude to Giovanni Forni and Carlos Matheus for listening to the arguments of the paper and their expository suggestions. We are grateful to Pascal Hubert for his interest in the project.  We thank Meysam Nassiri for the comments that made the exposition better. We appreciate interesting discussions with Boris Solomyak that motivated the result about substitutions in the appendix.

\section{Preliminaries}

In this section, we collect some definitions and tools we need in the forthcoming sections. 

\subsection{IETs and translation surfaces} \label{subsec: IETs and translation surfaces}


Let $d>1$ be an integer, $\mathcal{A}$ be a set of labels (or names) of cardinality $d$, and $\pi:=(\pi_t, \pi_b)$ be a $2$-tuple consisting of bijections $\pi_{\epsilon}: \mathcal{A} \to \{1, 2, \dots, d\}$, $\epsilon\in \{t, b\}$. Then to each length vector $\lambda \in \Delta^{\mathcal{A}}:= 
\{ \lambda \in \mathbb{R}_{+}^{\mathcal{A}}: \sum_{\alpha \in \mathcal{A}} \lambda_\alpha=1\} \subset \mathbb{R}_{+}^{\mathcal{A}}$, there associates an {\it{interval exchange transformation}} (IET) on the unit interval $I:=[0,1)$ given by the following
\[T_{\lambda, \pi}(x)=x-\sum\limits_{\pi_t(\alpha)<\pi_t(\alpha_0)}\lambda_{\alpha}+\sum\limits_{\pi_b(\alpha)
<\pi_b(\alpha_0)}\lambda_{\alpha}, \quad x\in I_{\alpha_0},\]
where $\alpha_0 \in \mathcal{A}$ and
 \[I_{\alpha_{0}}:=
\Big[\sum\limits_{\pi_t(\alpha)<\pi_{t}(\alpha_0)}\lambda_{\alpha}, \sum\limits_{\pi_t(\alpha) \leq \pi_t(\alpha_0)} \lambda_\alpha\Big).\]


A \textit{translation surface} is a surface $X$ obtained from a collection of polygons in the complex plane whose sides come in identified parallel pairs of equal length lying on the opposite sides of the polygons with respect to the natural orientation inherited from $\mathbb{C}$. Two translation surface structures are equivalent if one can be obtained from the other through a sequence of cutting and pasting of triangles. The identification of edges naturally induces an equivalence relation on the set of vertices. Each equivalence class is called a \textit{singularity} and the set of all singularities is denoted by $\Sigma.$ The Euclidean structure on $\mathbb{C}$ gives rise to a flat metric on the surface with singularities at elements of $\Sigma.$ For each $p\in \Sigma$, the total sum of internal angles of the vertices in the class of $p$ is an integer multiple $2\pi (m(p)+1)$ of $2\pi$, where $m(p) \in \mathbb{Z}_{\geq 0}$ is called the \textit{order} of the singularity $p$. Gauss-Bonnet implies the following relation between the genus $g$ of the surface $X$ and the order of singularities
\begin{equation}
    2g-2= \sum_{p\in \Sigma} m(p).
\end{equation}
The \textit{translation flow} associated to $X$ is obtained by moving from each point towards the north direction at unit speed with the caveat that for the finitely many trajectories that lead (in forward or backward time) to a singularity, the flow is not defined for all times.  \\

\noindent
\textbf{IETs as return maps of translation flows.} IETs can be naturally seen as the first return maps to horizontal transverse sections of translation flows on translation surfaces. Let 
\begin{equation}
    T^{+}(\pi):=\bigg\{\tau \in \mathbb{R}^{\mathcal{A}}: \sum_{\pi_{t}(\alpha)< \pi_{t}(\beta)} \tau_{\alpha}>0, \; \sum_{\pi_b(\alpha)> \pi_{b}(\beta)} \tau_{\alpha}<0\bigg\}.
\end{equation}
For $\tau \in T^+(\pi)$ and $\epsilon \in \{t, b\}$ we define the vectors 
\begin{equation}
    \xi^{\epsilon}_{\alpha}= \sum_{\beta: \pi_\epsilon(\beta) \leq \pi_\epsilon(\alpha)} \lambda_\beta+ i \tau_\beta \in \mathbb{C}.
\end{equation}
Now let $P(\lambda, \pi, \tau)$ be the $2d$-gon in the plane whose vertices are 
$\{0, \xi^{\epsilon}_{\alpha}, \epsilon \in \{t, b\}\}$. It is easy to see that the edges of such a polygon come in parallel pairs. Identifying the parallel pairs produces a translation surface $X(\lambda, \pi, \tau)$ with singularities corresponding to the vertices of $P$, which we denote by $\Sigma(\pi)$.

It is well-known since \cite{VeechMetricTheory} that the elements of $\Sigma(\pi)$ are in one-to-one correspondence with the cycles in the cycle decomposition of a certain permutation on $\{0, 1, \dots, d\}$, called the \textit{$\sigma$-permutation}, given by the following relations in terms of $\pi= \pi_b \circ \pi^{-1}_t$:
\begin{equation}
\sigma(i)=
\begin{cases}
    0 & \text{if} \;\;\; 0 \leq i<d,\;\pi(i+1)=1,\\
    \pi^{-1}(\pi(i+1)-1), & \text{if} \;\;\; 0\leq i<d,\;\pi(i+1)\neq 1,\\
    \pi^{-1}(\pi(d)+1)-1 &\text{if} \;\;\; i=d.
\end{cases}
\end{equation}
$\Sigma(\pi)$ has a distinguished element corresponding to $0$, or what is the same, the leftmost vertex of $P(\lambda, \pi, \tau)$. We denote this element by $S_0$.\\




\noindent
\textbf{Zippered rectangles.} 
\textit{Zippered rectangles} is the name of a construction introduced by Veech that creates translation surfaces through suspension of IETs with locally constant roof functions.

Let $\Omega_\pi$ be the following $\mathcal{A}\times \mathcal{A}$ matrix.

\begin{equation}
    (\Omega_{\pi})_{(\alpha, \beta)}:= 
    \begin{cases}
    +1 & \text{if} \; \pi_{b}(\beta)< \pi_b(\alpha), \; \pi_t(\beta)> \pi_t(\alpha),\\
    -1 & \text{if} \; \pi_b(\beta)> \pi_b(\alpha), \; \pi_t(\beta)< \pi_t(\alpha),\\
    0 & \text{otherwise},
    \end{cases}
\end{equation} 



We explain here how the surface $X(\lambda, \pi, \tau)$ can be obtained via identifications of a union of rectangles. Let $h= -\Omega_{\pi}(\tau) \in H^{+}(\pi)$

The surface $X(\lambda, \pi, \tau)$ may alternatively be obtained by applying certain identifications on the union of rectangles of the 
form $I_{\alpha}\times [0,h_\alpha]$. Roughly speaking, the data $\tau$ determine the heights up to which two adjacent rectangles are zipped together in the vertical direction. Regarding the horizontal sides, we 
identify the top edge of the rectangle $I_\alpha \times [0, h_\alpha]$ with the subinterval $T_{\lambda, \pi}(I_\alpha)$ of $I$. For a detailed explanation of this construction and its consequences, we refer the reader to \cite{VeechGaussmeasure}, or \cite{viana2008dynamics}.  

The identifications in the above constructions are given by translations and induce a translation structure on $X(\lambda, \pi, \tau)$ which is equivalent to the one defined previously. The $1$-st homology space $H_1(X, \mathbb{R})$ corresponds to the subspace $H(\pi)$ of dimension $2g$ where $g$, the genus of $X$ is only a function of $\pi$.

If $\kappa:=|\Sigma(\pi)|$ is the number of singularities of $X$ and $N(\pi)$ the kernel of $\Omega_\pi$, then Veech \cite{VeechGaussmeasure} showed that 
\begin{equation} \label{genus and singularities}
    \dim(N(\pi))= \kappa -1, \; d= 2g+ \kappa-1.
\end{equation}

\subsection{Rauzy-Veech and Zorich renormalizations}\label{sec:renormalization}
\smallskip
Let $T_{\lambda^{(0)}, \pi^{(0)}}$ be an IET on the interval $I^{(0)}$ with parameters $(\lambda^{(0)}, \pi^{(0)})$
and $\alpha_{t}, \alpha_{b}$ be the letters with the property that $\pi^{(0)}_{t}(\alpha_{t})=\pi^{(0)}_{b}(\alpha
_{b})=d$. The Rauzy induction algorithm gives a relation between the length vector and combinatorial data of $T_{\lambda^{(0)}, \pi^{(0)}}$ and those of its first return map to the subinterval $I^{(1)}:=[0, 1- \min \{\lambda_{\alpha_t}, \lambda_{\alpha_b}\})$ under the assumption that $\lambda^{(0)}_{\alpha_t} \neq \lambda^{(0)}_{\alpha_b}$.
{
 The resulting transformation $T_{\lambda^{(1)}, \pi^{(1)}}$ is defined as follows
\begin{itemize} 
\item[i)] $\lambda^{(0)}_{\alpha_{t}}> \lambda^{(0)}_{\alpha_{b}}$: then let $\lambda^{(1)}_{\alpha}:=\lambda^{(0)}_{\alpha}$ for all $\alpha \neq \alpha_{t}$, $\lambda^{(1)}_{\alpha_{t}}:= \lambda^{(0)}_{\alpha_{t}}- \lambda^{(0)}_{\alpha_{b}}$, $\pi_{t}^{(1)}:=\pi_{t}^{(0)}$, and
\begin{equation} \label{eq:top}
(\pi^{(1)}_{b})^{-1}(i):=
\begin{cases}
(\pi^{(0)}_{b})^{-1}(i) & \text{if}\; i \leq \pi^{(0)}_{b}(\alpha_t),\\
\alpha_{b} & \text{if}\;
i=\pi^{(0)}_{b}(\alpha_t)+1,\\
(\pi^{(0)}_{b})^{-1}(i-1) & \text{otherwise},
\end{cases}
\end{equation}
in which case $(\lambda^{(0)}, \pi^{(0)})$ is said to be of \textit{top type}
(since $\alpha_t$ ``wins'' against $\alpha_b$).
\item[ii)]$\lambda_{\alpha_{b}}^{(0)}>\lambda_{\alpha_{t}}^{(0)}$: then let $\lambda_{\alpha}^{(1)}:=\lambda_{\alpha}^{(0)}$ for all $\alpha \neq \alpha_{b}$, $\lambda_{\alpha_{b}}^{(1)}:=\lambda_{\alpha_{b}}^{(0)}- \lambda_{\alpha_{t}}^{(0)}$, $\pi_{b}^{(1)}:=\pi_{b}^{(0)}$, and 
\begin{equation}\label{eq:bottom}
(\pi^{(1)}_{t})^{-1}(i):=
\begin{cases}
(\pi^{(0)}_{t})^{-1}(i) & \text{if}\; i \leq \pi^{(0)}_{t}(\alpha_b),\\ 
\alpha_{t} & \text{if}\; i=\pi^{(0)}_{t}(\alpha_b)+1,\\
(\pi^{(0)}_{t})^{-1}(i-1) & \text{otherwise},
\end{cases}
\end{equation}
in which case $(\lambda^{(0)}, \pi^{(0)})$ is said to be of \textit{bottom type}
(since $\alpha_b$ ``wins'' against $\alpha_t$).
\end{itemize}
\begin{figure}
    \centering
    \includegraphics[width=1\linewidth]{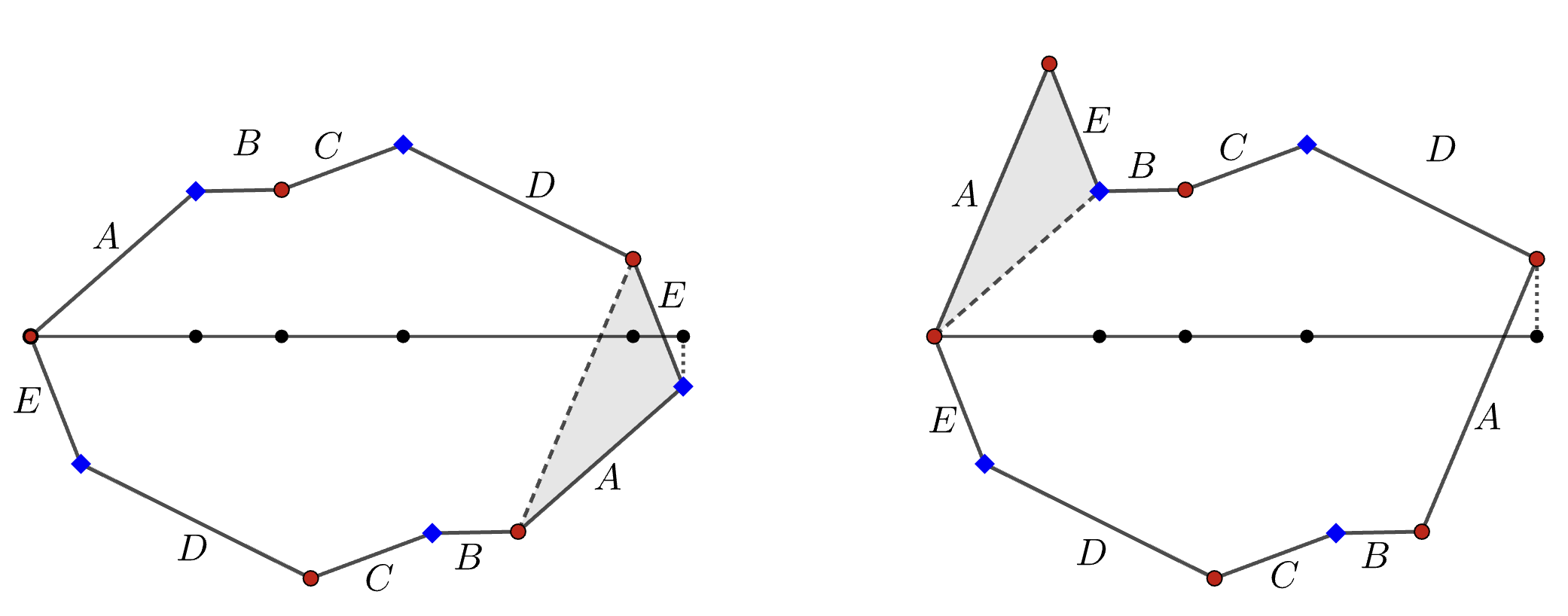}
    \caption{Renormalization of bottom type transforming the permutation $\left(\begin{smallmatrix}
        ABCDE\\
        EDCBA
    \end{smallmatrix}\right)$
    to the permutation 
    $\left(\begin{smallmatrix}
        AEBCD\\
        EDCBA
    \end{smallmatrix}\right).$}
    \label{fig1}
\end{figure}

\noindent
\textbf{Notation.} Henceforth, we denote the image of any object under one step of renormalization by the same symbol with an additional superscript $(1)$. For instance, the image of $(\lambda, \pi, \zeta)$ is denoted by $(\lambda^{(1)}, \pi^{(1)}, \zeta^{(1)})$.\\

The combinatorial data $\pi$ is \textit{irreducible}, 
if there exists no $1\leq k < d$ such that 
$\pi_{t}^{-1}(\{1,\dots,k\})=\pi_{b}^{-1}(\{1,\dots,k\})$. We denote the space of such combinatorial data by $\mathfrak{S}^{0}(\mathcal{A})$. The above operations define an equivalence relation that partitions
the set of combinatorial data into equivalence \textit{Rauzy classes}, denoted by $\mathfrak{R} \subset \mathfrak{S}^{0}(\mathcal{A})$. 

Let $\mathfrak{R}$ be a Rauzy class and $\pi \in \mathfrak{R}$ a permutation. We define $\mathcal{Q}_{R}: \mathbb{R}_{+}^{\mathcal{A}} \times \mathfrak{R} \to 
\mathbb{R}_{+}^{\mathcal{A}} \times \mathfrak{R}$, called \textit{Rauzy 
induction map}, to be the map that sends $(\lambda^{(0)}, \pi^{(0)})$ to $(\lambda^{(1)}, \pi^{(1)})$, according to equations $\eqref{eq:top}$ and $\eqref{eq:bottom}$. Rescaling $\lambda^{(1)}$ back to size $1$ 
(size being the $\ell^{1}$-norm) yields a transformation $\mathcal{R}_{R}: 
\Delta^{\mathcal{A}} \times \mathfrak{R} \to \Delta^{\mathcal{A}} \times \mathfrak{R}$, which we call the \textit{Rauzy renormalization map}.

Zorich defined an accelerated version of the Rauzy renormalization, which is now called \textit{Zorich transformation}.

\begin{definition}
For an element $(\lambda, \pi) \in \Delta^{\mathcal{A}} \times \mathfrak{R}$, let    $n:=n(\lambda, \pi)$ be the smallest 
 $n\in \mathbb{N}$ for which the type of $(\lambda^{(n)}, \pi^{(n)})=\mathcal{Q}_{R}^{n}(\lambda, \pi)$ is different from that of $(\lambda, \pi)$. Then, the Zorich induction, $\mathcal{Q}_{Z}(\lambda, \pi)$ is defined by 
\begin{equation}
    \mathcal{Q}_{Z}(\lambda, \pi):=(\lambda^{(n)}, \pi^{(n)}),
\end{equation} 
and the Zorich renormalization is defined by rescaling the intervals back to size $1$
\begin{equation}
\mathcal{R}_Z(\lambda,\pi):=\big(\frac{\lambda^{(n)}}{|\lambda^{(n)}|},\pi^{(n)}\big).
\end{equation}
\end{definition}
Zorich showed the following
\begin{thm} (\cite{ZorichGaussmeasure}) Let $\mathfrak{R} \subset \mathfrak{S}^{0}(\mathcal{A})$ be a Rauzy class. Then $\mathcal{R}_{Z}|_{\Delta^{\mathcal{A}} \times \mathfrak{R}}$ admits a unique ergodic invariant probability measure $\mu_{Z}$ equivalent to the Lebesgue measure.
\end{thm}

\subsection{Rauzy-Veech and Zorich cocycles}
The {\it Rauzy-Veech cocycle} is a linear cocycle that gives the action of Rauzy induction on the roof functions in the zippered rectangles construction, that is, for $(\lambda,\pi)$, $B^{R}(\lambda, \pi)$ is the $\mathcal{A}\times \mathcal{A}$ matrix given by the following formula
\begin{equation}
    B^{R}(\lambda, \pi):=
    \begin{cases}
    I+ E_{\alpha_{b}\alpha_{t}} & \text{if} \; (\lambda, \pi)\; \text{is of top type,}\\
    I+ E_{\alpha_{t}\alpha_{b}} & \text{if} \; (\lambda, \pi) \; \text{is of bottom type,}
    \end{cases}
\end{equation}
where for all $\alpha, \beta \in \mathcal{A}$ 
the symbol $E_{\alpha\beta}$ denotes the $\mathcal{A}\times \mathcal{A}$ elementary matrix with a single non-zero entry equal to $1$ in position $(\alpha,\beta)$.

With the above definition, $(\mathcal{R}_{R}, B^{R})$ forms an integral
cocycle. The corresponding cocycle over the Zorich transformation is called the {\it Zorich cocycle} and is denoted by $(\mathcal{R}_{Z},B^{Z})$. \\


\noindent
\textbf{Invariant structures.} 
Veech \cite{VeechGaussmeasure} showed that the Rauzy induction extends to an invertible dynamical system 
(natural extension), called {\it Rauzy--Veech induction}, on the space of zippered rectangles. Moreover, the locally constant subbundles $H(\pi), N(\pi)$ are respectively invariant and covariant under the {\it Rauzy--Veech cocycle}. That is, for $\mathcal{R}_{R}(\lambda, \pi)=(\lambda^{(1)}, \pi^{(1)})$, we have
\begin{align}
    B^{R}(\lambda, \pi)\cdot H(\pi)=H(\pi^{(1)}), \\
    {}^t\!B^{R}(\lambda, \pi)\cdot N(\pi^{(1)})=N(\pi)\,.
\end{align}

The matrix $\Omega_\pi$ provides a symplectic intersection pairing on $H(\pi)$ (corresponding to the Homology intersection pairing on $H_1(X, \mathbb{R})$). It is well known that the Zorich cocycle preserves $\Omega_{\pi}$, that is, 
\begin{equation}
    B^{R}(\lambda, \pi) \Omega_{\pi}{}^t\!B^{R}(\lambda, \pi)= \Omega_{\pi^{(1)}}.
\end{equation}
This implies that the spectrum of the Zorich cocycle restricted to the subspace $H(\pi)$ is symmetric with respect to zero. 

\subsection{Toral and twisted cocycles} \label{sec: Twisted Cocycle}
The main subject of our study, the twisted cocycle, is a complex cocycle defined over the toral version of the Rauzy-Veech and Zorich cocycles. The Rauzy-Veech and Zorich cocycles being integer-valued naturally induce an action by linear automorphisms on the $d$-torus $\mathbb{T}^d= \mathbb{R}^d/\mathbb{Z}^{d}$, which we call the \textit{toral 
Rauzy--Veech cocycle} and the \textit{toral Zorich cocycle}, respectively. So the toral Rauzy-Veech cocycle is given by 
\begin{align}
 \toral^{R}:&\Delta \times \mathfrak{R} \times \mathbb{T}^{d} \longrightarrow  \Delta \times \mathfrak{R} \times \mathbb{T}^{d},
\\
& (\lambda,\pi, \zeta)\longmapsto (\mathcal{R}_{R}(\lambda, \pi), B^R(\lambda, \pi)\zeta).
\end{align}
and its accelerated version by \[\toral^Z(\lambda,\pi, \zeta):= (\mathcal{R}_{Z}(\lambda, \pi), B^Z(\lambda, \pi)\zeta).\] 
For an element $(\lambda, \pi, \zeta) \in \Delta \times \mathfrak{R} \times \mathbb{T}^{d}$, we define the corresponding twisted Rauzy-Veech matrix by
\begin{equation} 
    \tw^{R}(\lambda, \pi, \zeta):=
    \begin{cases}
    I+ \exp(2\pi i \zeta_{\alpha_b})E_{\alpha_{b}\alpha_{t}}& \text{if} \; (\lambda, \pi)\; \text{is of top type},\\
    I+ E_{\alpha_{t}\alpha_{b}} +(-1+ \exp(2\pi i 
    \zeta_{{\alpha}_b})) E_{\alpha_t \alpha_t} & \text{if} \; (\lambda, \pi) \; \text{is of bottom type.}
    \end{cases}
\end{equation}
The corresponding twisted Zorich matrix is defined by
\begin{equation}
    \tw^{Z}(\lambda, \pi, \zeta):= \prod\limits_{k=0}^{n(\lambda, \pi)-1} \tw^{R}((\toral^{R})^{j}(\lambda, \pi, \zeta)),
\end{equation}
and the corresponding twisted Zorich cocycle is  $(\mathcal{R}^Z_{\mathbb{T}^d},\mathcal{B}^Z)$.  
The definition of the twisted Rauzy-Veech matrix depends solely upon the equivalence class of $\zeta$ mod $\mathbb{Z}^{d}$.\\

\noindent
\textbf{Invariant measures.}
Let $k>1$ be an arbitrary integer and $Q_k \subset \mathbb{T}^d$ be the set of non-zero rational points of denominator $k$, that is, 
\begin{equation}\label{eq:def-Qp}
Q_k:=\bigg\{\frac{\n}{k} \; : \; \n \in \mathbb{Z}^d, \n\neq 0\bigg\} \subset \mathbb{R}^d/\mathbb{Z}^d.   
\end{equation}

The subspace $H(\pi)$ is defined over $\mathbb{Q}$ and therefore $H(\pi) \cap \mathbb{Z}^d \subset H(\pi)$ is a lattice. We denote by $\mathbb{T}^{2g}_{\pi}$ the torus obtained by quotienting $H(\pi)$ by this lattice. We let $Q^{(\pi)}_{k} \subset H(\pi)/(H(\pi) \cap \mathbb{Z}^d)$ be the set of nonzero rational points with denominator $k$ in $\mathbb{T}^{2g}_\pi$. We let $\nu_{k}, \nu^{(\pi)}_{k}$ be the atomic measures that give equal weight to the elements of $Q_k$ and $Q^{(\pi)}_{k}$, respectively. Let $m_{\mathbb{T}^{2g}_{\pi}}$ be the Lebesgue measure on $\mathbb{T}^{2g}_{\pi}$ normalized to be a probability measure. It can be readily verified that the measures
\begin{itemize}
    \item $\sum_{\pi \in 
    \mathfrak{R}}\mu_{Z}\big\vert_{\Delta_{\pi}} \times m_{\mathbb{T}^{2g}_{\pi}}$, $\mu_Z \times m_{\mathbb{T}^d}$, 
    \item $\sum_{\pi \in 
    \mathfrak{R}}\mu_{Z}\big\vert_{\Delta_{\pi}} \times \nu^{(\pi)}_k$, $\mu_Z \times \nu_k,$
\end{itemize}
are invariant under $\toral^{Z}$.

\section{Invariant structures}\label{sec:inv-structure}
This section is devoted to the study of algebraic structures invariant under the twisted cocycle. We construct invariant and covariant sections as well as invariant Hermitian forms. A part of this section is devoted to the construction of a real cousin of the twisted cocycle that is proved to admit an invariant symplectic structure. 






\subsection{Invariant and covariant sections} 
We observe that the "discrete twisted derivative" of the function $1$ for the interval exchange $T_{\lambda, \pi}$ gives an invariant section for the action of the twisted cocycle. Below for any vector $\zeta \in \mathbb{R}^\mathcal{A}$, we let $z_\alpha=\exp(2\pi i \zeta_\alpha)$.

\begin{lemma}\label{lem:inv-section}
    Let $d>1$ and $\pi=(\pi_t, \pi_b)$ be an irreducible $d-$permutation. Then there exists a section $s_{\zeta} \in \mathbb{C}^{\mathcal{A}}$ given by
    \begin{equation}\label{eq:def-inv-bundle}
    (s_{\zeta})_{\alpha}= 1- z_\alpha,
    \end{equation}
    that satisfies the following invariance condition
    \begin{equation}
        \mathcal{B}(\lambda, \pi, \zeta).s_{\zeta}= s_{\zeta^{(1)}}.
    \end{equation}
\end{lemma}
\begin{proof}
    The proof follows from a simple computation which we leave to the reader. 
\end{proof}
\begin{rmk}
    After explaining our observation regarding the existence of the invariant section in the setting of IETs to Boris Solomyak, he pointed out that this observation can be extended to substitution systems as well (see appendix, \cite{solomyak_twisted_cocycle}).
\end{rmk}

Let $W_{\zeta^{(1)}}= \langle s_{\zeta^{(1)}}\rangle^{\perp}$ be the subspace perpendicular to the invariant section. The invariance of $s_\zeta$ under $\mathcal{B}(\lambda, \pi, \zeta)$ implies the invariance of $W_{\zeta}$ under $\mathcal{B}^{*}(\lambda, \pi, \zeta)$, that is,
\begin{equation}\label{eq:inv-W}
\mathcal{B}^{*}(\lambda, \pi, \zeta)W_{\zeta^{(1)}}=W_{\zeta}.
\end{equation}


\noindent
\textbf{Standing assumption.} From now on we assume that $\zeta\in H(\pi) \setminus \mathbb{Z}^d$ does not belong to the integer lattice and let $\eta \in \mathbb{R}^\mathcal{A}$ be such that $-\Omega_\pi \eta= \zeta$ (when $\eta$ is not unique we pick one arbitrarily).

For every $\beta \in \mathcal{A}$ we let 
\begin{equation}
    \w{\beta}:= \exp(-2\pi i \eta_\beta),
\end{equation}
and define $c_{\pi,\eta}:\{0\} \cup \mathcal{A} \to \mathbb{C}$ by
\begin{equation}
 c_{\pi,\eta}(\alpha):= \prod\limits_{\pi_t(\beta)\leq \pi_t(\alpha)}\w{\beta}^{-1}.
\end{equation}

\begin{figure}[h]
    \centering
    \includegraphics[width=0.5\linewidth]{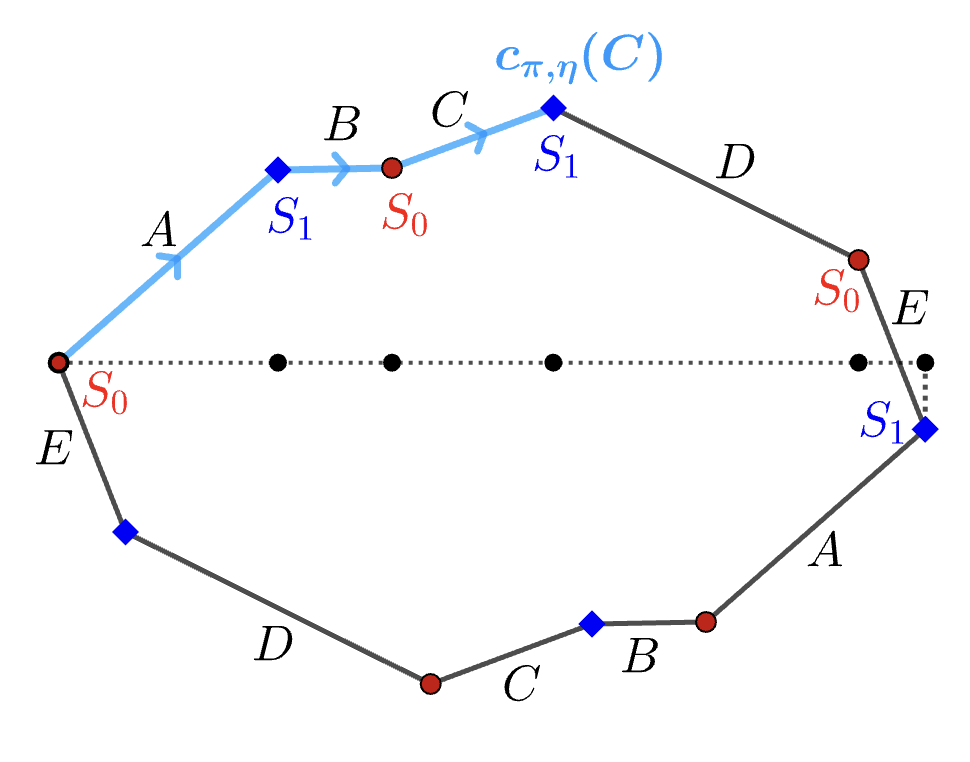}
    \caption{$S_0$ is the distinguished singularity and $c_{\pi, \eta}(C)$ is the amount of twist (induced by $-\eta$) along the colored edges.}
    \label{fig2}
\end{figure}

Given a cycle $S$ in the cycle decomposition of the permutation $\sigma$, let $\vs \in \mathbb{C}^{\mathcal{A}}$ be given by 
\begin{equation}\label{eq:def-v_S}
    \vs(\alpha)= c_{\pi, \eta}(\prealpha)\one_S(\prealpha)- c_{\pi, \eta}(\alpha)\one_S(\alpha),
\end{equation}
where $\alpha^{-1}$ is the symbol of the interval immediately to the left of $I^{t}_{\alpha}$, and is $0$ when $\pi_t(\alpha)=1$. When $S=S_0$ is the distinguished element of $\Sigma(\pi)$, we drop the superscript and denote the associated vector by $v_{\pi,\zeta}$.
\begin{lemma}\label{lem:inv-v_S} Let $\eta , \eta' \in \mathbb{R}^\mathcal{A}$ be such that $\eta-\eta' \in N(\pi)$, then for any cycle $S$ in the cycle decomposition of $\sigma$, there exists $u(\eta, \eta', S) \in \mathbb{C}$ of unit norm such that $\vs= u(\eta, \eta', S) v^{S}_{\pi,\eta'}$. 
\end{lemma}
\begin{proof}
    Assume that $\alpha, \beta \in S$ are two symbols with $\pi_t(\alpha) < \pi_t(\beta)$. The equality $-\Omega_\pi(\eta)=-\Omega_{\pi}(\eta')$ implies that 
    \begin{equation}
        \sum_{\pi_t(\alpha)\leq \pi_t(\theta) \leq \pi_t(\beta)} \eta_\theta=  \sum_{\pi_t(\alpha)\leq \pi_t(\theta) \leq \pi_t(\beta)} \eta'_\theta
    \end{equation}
    and therefore 
    \begin{equation}
        \frac{c_{\pi,\eta}(\beta)}{c_{\pi,\eta}(\alpha)}= \frac{c_{\pi, \eta'}(\beta)}{c_{\pi,\eta'}(\alpha)}.
    \end{equation}
    Taking $u(\eta, \eta',S)= \frac{c_{\pi,\eta}(\alpha)}{c_{\pi,\eta'}(\alpha)}$ for any $\alpha \in S$, gives the desired result.  
\end{proof}

\begin{prop}(Covariant sections)\label{prop:N-span}
    Let $d>1$ be a positive integer and $\mathfrak{R}$ and $\pi=(\pi_t,\pi_b)$ a $d$-permutation. Assume that $\zeta, \eta, \vs$ are as above. Then
    \begin{equation}\label{eq:covariant-sections}
         (\mathcal{B}(\lambda, \pi, \zeta)^{*})^{-1} \vs= v^{S^{(1)}}_ {\pi^{(1)},\eta^{(1)}}.
    \end{equation}
\end{prop}
\begin{proof}
We treat the case of bottom and top operations separately. \\

\noindent
\textbf{Top operation.} In this case
\(
\mathcal{B}(\lambda, \pi, \zeta)= Id+ z_{\alpha_b} E_{\alpha_b \alpha_t},   
\) and therefore 
\begin{equation}
(\mathcal{B}^{*})^{-1}= Id-{z^{-1}_{\alpha_b}} E_{\alpha_t \alpha_b}.
\end{equation}
Recall that $\eta^{(1)}={}^{t}B(\lambda, \pi)^{-1}\eta$ and observe that
\begin{equation} \label{change of coefficients}
    c_{\pi^{(1)}, \eta^{(1)}}(\alpha)=
\begin{cases}
    c_{\pi, \eta}(\alpha) &  \alpha \neq \alpha_t,  \\
    c_{\pi, \eta}(\alpha) \w{\alpha_{b}} & \alpha=\alpha_t.
\end{cases}
\end{equation}
Notice that since the left endpoint of the $\alpha_b$-interval is identified to the right endpoint of the new $\alpha_t$-interval we have
\begin{equation} \label{change of characteristics}
\one_{S^{(1)}}(\alpha)=
\begin{cases}
\one_{S}(\alpha) & \text{if} \; \alpha \neq \alpha_t,\\
\one_S(\prealpha_b) & \text{if} \; \alpha= \alpha_t.
\end{cases}    
\end{equation}
So 
\begin{itemize}
    \item if $\alpha \neq \alpha_t$ we get
\begin{equation}
     v^{S^{(1)}}_{\pi^{(1)}, \eta^{(1)}}(\alpha)=c_{\pi^{(1)}, \eta^{(1)}}(\prealpha) \one_{S^{(1)}}(\prealpha)-c_{\pi^{(1)}, \eta^{(1)}}(\alpha) \one_{S^{(1)}}(\alpha) = \vs{(\alpha)}.
\end{equation}
\item Next, if $\alpha =\alpha_t$, then by \eqref{change of coefficients}, \eqref{change of characteristics} we have
\begin{align}
    v^{S^{(1)}}_{\pi^{(1)},\eta^{(1)}}(\alpha_t) & = c_{\pi, \eta} (\prealpha_t )\one_{S}(\prealpha_t)- \w{\alpha_b}c_{\pi, \eta}(\alpha_t)\one_S(\prealpha_b )\\
    & = \vs(\alpha_t)+ c_{\pi, \eta}(\alpha_t)\one_S(\alpha_t)-\w{\alpha_b}c_{\pi, \eta}(\alpha_t)\one_S(\prealpha_b).
\end{align}
By the fact that  $\one_S(\alpha_b)=\one_S(\alpha_t)$ and by the definition of $\vs$ 
\begin{equation}
    \vs(\alpha_b)= c_{\pi, \eta}(\prealpha_b)\big(\one_S(\prealpha_b)-\w{\alpha_b}^{-1}\one_S(\alpha_t)\big).
\end{equation} 
Moreover,  
\begin{equation}
    z_{\alpha_b}= c_{\pi, \eta}( \prealpha_b) \times \prod\limits_{\alpha\neq \alpha_b}\w{\alpha}  = c_{\pi, \eta}( \prealpha_b)c_{\pi, \eta}(\alpha_t)^{-1}\w{\alpha_b}^{-1}, 
\end{equation}
and therefore 
\begin{equation}
    v^{S^{(1)}}_{\pi^{(1)}, \eta^{(1)}}(\alpha_t)= \vs(\alpha_t)-z^{-1}_{\alpha_b}\vs(\alpha_b),
\end{equation}
\end{itemize}
which implies equation \eqref{eq:covariant-sections}.

\noindent
\textbf{Bottom operation.} Notice that in this case \(
    \mathcal{B}(\lambda, \pi, \zeta)= Id+ (z_{\alpha_b}-1) E_{\alpha_t \alpha_t}+E_{\alpha_t\alpha_b},
\) and
\begin{equation}
 (\mathcal{B}^{*})^{-1}= Id+(z_{\alpha_b}-1) E_{\alpha_t \alpha_t}-z_{\alpha_b}E_{\alpha_b\alpha_t}.  
\end{equation}
It is easy to verify that 
\begin{equation} 
    c_{\pi^{(1)}, \eta^{(1)}}(\alpha)=
\begin{cases}
    c_{\pi, \eta}(\alpha) &  \alpha \notin \{\alpha_b, \alpha_t\},  \\
    c_{\pi, \eta}(\alpha_b)   \w{\alpha_t} &  \alpha =\alpha_b,\\
    c_{\pi, \eta}(\alpha_b) & \alpha=\alpha_t.
\end{cases}
\end{equation}
As the left endpoint of the $\alpha_t$-edge is identified to the right endpoint of the new $\alpha_b$-edge, we have
\begin{equation}
\one_{S^{(1)}}(\alpha)=
    \begin{cases}
        \one_{S}(\alpha) & \text{if} \;  \alpha \notin \{ \alpha_b, \alpha_t\}\\
        \one_S(\prealpha_t)  & \text{if} \;  \alpha =\alpha_b\\
        \one_S(\alpha_b) & \text{if} \;  \alpha =\alpha_t\\
    \end{cases}
\end{equation} 

So 
\begin{itemize}
    \item if $\alpha \notin\{ \alpha_b,\alpha_t\}$, then
\begin{equation}
    v^{S^{(1)}}_{\pi^{(1)}, \eta^{(1)}}(\alpha)=c_{\pi^{(1)}, \eta^{(1)}}(\prealpha) \one_{S^{(1)}}(\prealpha)-c_{\pi^{(1)}, \eta^{(1)}}(\alpha) \one_{S^{(1)}}(\alpha) = \vs{(\alpha)}.
\end{equation}
\item If $\alpha=\alpha_t$, then 
\begin{align}
    v^{S^{(1)}}_ {\pi^{(1)}, \eta^{(1)}}(\alpha_t) & = c_{\pi^{(1)}, \eta^{(1)}} (\alpha_b )\one_{S^{(1)}}(\alpha_b)- c_{\pi^{(1)}, \eta^{(1)}}(\alpha_t)\one_{S^{(1)}}(\alpha_t)\\
    & = c_{\pi, \eta}(\alpha_b)\big(  \w{\alpha_t} \one_S(\prealpha_t)-  \one_S(\alpha_t)\big).
\end{align}
Notice that $\vs(\alpha_t)= c_{\pi, \eta}(\alpha_t)\w{\alpha_b}  \one_S(\prealpha_t)-\one_S(\alpha_t)$ and $z_{\alpha_b}= c_{\pi, \eta}( \alpha_b) \times c_{\pi, \eta}(\alpha_t)^{-1}$. Hence  
\begin{equation}
v^{S^{(1)}}_{\pi^{(1)},\eta^{(1)}}(\alpha_t)= z_{\alpha_b}\vs (\alpha_t).
\end{equation}
\item Else if $\alpha=\alpha_b$, then 
 as $\one_S(\alpha_t)=\one_S(\alpha_b)$
\begin{align}
    v^{S^{(1)}}_{\pi^{(1)}, \eta^{(1)}}(\alpha_b) & = c_{\pi^{(1)}, \eta^{(1)}} (\prealpha_b)\one_{S^{(1)}}(\prealpha_b)- c_{\pi^{(1)}, \eta^{(1)}}(\alpha_b)\one_{S^{(1)}}(\alpha_b)\\ 
    &= c_{\pi, \eta}(\prealpha_b) \big( \one_S(\prealpha_b)-\w{\alpha_t}\w{\alpha_b}^{-1}\one_S(\prealpha_t)\big) \\
    &= c_{\pi, \eta}(\prealpha_b)\big( \one_S(\prealpha_b)- \w{\alpha_b}^{-1}\one_S(\alpha_b)\big) \\
    & + c_{\pi, \eta}(\alpha_b)\big(\one_S(\alpha_t)- \w{\alpha_t}\one_S(\prealpha_t)\big) \\
    &= \vs(\alpha_b) - z_{\alpha_b} \vs(\alpha_t),
\end{align}
\end{itemize}
proving the desired equality.
\end{proof}

We now define the complex subspaces $N(\pi, \zeta)$ and $\widetilde{N}(\pi, \zeta)$ by
\begin{equation}
    N(\pi, \zeta):=\mathrm{Span}_{\mathbb{C}}\big\{ \vs| S\in \Sigma(\pi)\setminus\{S_0\}\big\}, \quad \widetilde{N}(\pi,\zeta):= N(\pi, \zeta) \oplus \mathbb{C}v_{\pi, \zeta}.
\end{equation}
The fact that the sum above is direct will be proved in Lemma \ref{lem:properties-of-omega}. 

\subsection{Invariant Hermitian and symplectic forms}

In this subsection, we introduce a family of matrices $\Omega_{\pi, \zeta}$ (which are going to be treated as $2$-forms) satisfying a certain invariance property under the twisted cocycle (see \eqref{eq:inv-big-omega} and \eqref{eq:inv-small-omega}). This in turn allows us to establish the existence of invariant symplectic structures under the twisted cocycle, implying the symmetry of its Lyapunov spectrum with respect to zero.





Let $\Omega_{\pi,\zeta}$ be defined by the following formula
\begin{equation}\label{eq:def-Omega}
 \Omega_{\pi, \zeta}(\alpha, \beta):=
\begin{cases}
   0 &  \text{if}\;\;\;  \alpha=\beta,\\
   0 & \text{if}\;\;\;\pi_t(\beta) >\pi_t(\alpha), \pi_b(\beta)> \pi_b(\alpha),\\
   1 & \text{if}\;\;\;  \pi_t(\beta)<\pi_t(\alpha), \pi_b(\beta)>\pi_b(\alpha),\\
   -z_\alpha z^{-1}_\beta & \text{if} \;\;\; \pi_t(\beta)>\pi_t(\alpha), \pi_b(\beta)< \pi_b(\alpha),\\
   1-z_\alpha z^{-1}_\beta & \text{if} \;\;\;\pi_t(\beta)<\pi_t(\alpha), \pi_b(\beta)< \pi_b(\alpha).
\end{cases}
\end{equation}

\begin{prop} \label{inv-Omega}(Invariance of the family $\Omega_{\pi, \zeta}$) 
The following identity 
    \begin{equation}\label{eq:inv-big-omega}
    \mathcal{B}(\lambda, \pi, \zeta) \Omega_{\pi, \zeta} \mathcal{B}^{*}(\lambda, \pi, \zeta)= \Omega_{\pi^{(1)}, \zeta^{(1)}},
\end{equation}
holds on the subspace $W_{\zeta^{(1)}}$. 
\end{prop}
\begin{proof} Let $\epsilon \in \{t, b\}$ be the type of the renormalization operation (top or bottom) and $\epsilon'$ be such that $\{\epsilon, \epsilon'\}=\{t, b\}$. Then, a case-by-case calculation shows that 
\begin{equation}
     \mathcal{B}(\lambda, \pi, \zeta) \Omega_{\pi, \zeta} \mathcal{B}^{*}(\lambda, \pi, \zeta)-\Omega_{\pi^{(1)}, \zeta^{(1)}}= z_{\alpha_b}e_{\alpha_{\epsilon'}}s^{*}_{\zeta^{(1)}
     },
\end{equation}
where $e_{\alpha}$ stands for the $d$-by-$1$ matrix whose only nonzero coordinate is its $\alpha$-th coordinate which is equal to one. 
\end{proof}

\begin{lemma}\label{lem:properties-of-omega}
$\Omega_{\pi, \zeta}$ satisfies the following properties:
\begin{enumerate}
    \item Skew symmetry restricted to $W_\zeta$
    \begin{equation}\label{Omega+Omega*}
        (\Omega_{\pi, \zeta}+\Omega^{*}_{\pi, \zeta} )W_{\zeta} \subset \langle s_{\zeta} \rangle,
    \end{equation} 
    \item $\Omega_{\pi, \zeta} v_{\pi,\zeta}= s_{\zeta}$, 
    \item $N(\pi, \zeta)= Ker \Omega_{\pi, \zeta}$ is a $(\kappa-1)$-dimensional subbundle invariant under $(\mathcal{B}^*)^{-1}$,
    \item $\widetilde{H}(\pi, \zeta):= \Omega_{\pi, \zeta} W_{\zeta}$ is a $(2g-1)$-dimensional subbundle invariant under $\mathcal{B}$,
     \item 
     $\widetilde{N}(\pi, \zeta)=N(\pi,\zeta)\oplus \langle v_{\pi,\zeta}\rangle =\widetilde{H}(\pi,\zeta)^\perp$ is a $\kappa$-dimensional subbundle invariant under $(\mathcal{B}^*)^{-1}$.
\end{enumerate}
\end{lemma}

\begin{proof}
Let $z \in \mathbb{C}^\mathcal{A}$ be the column vector whose $\alpha$-th coordinate. From the definition of $\Omega_{\pi, \zeta}$ it follows that 
\begin{equation}
    \Omega_{\pi,\zeta}+ \Omega_{\pi,\zeta}^{*}= J-z.z^{*},
\end{equation}
where $J$ is the matrix of all ones. So now if $\nu\in W_{\zeta}$ then 
\begin{equation}
     (\Omega_{\pi,\zeta}+ \Omega_{\pi,\zeta}^{*})\nu= (J-z.z^{*})\nu
\end{equation}
and the $\alpha$-th coordinate of the above vector would be
\begin{equation}\label{eq: w calculation}
    \sum_{\beta}\nu_\beta-z_\alpha(\sum_{\beta} \nu_\beta z_\beta ^{-1}).
\end{equation}
As $\nu \in W_{\zeta}$ we have $\sum_{\beta} \nu_\beta(1-z_\beta^{-1})=0$. Thus the right-hand side of equation \eqref{eq: w calculation} can be further simplified to $(\sum_{\beta} \nu_\beta)(1-z_\alpha)=(\sum_{\beta} \nu_\beta)s_\zeta(\alpha)$. Repeating the argument for all $\alpha$ finishes the proof of the first statement. 

For the proof of part $2$ and the inclusion $N(\pi, \zeta) \subset Ker \Omega_{\pi, \zeta}$ we observe that 
\begin{equation}\label{eq:Omega-v=S}
\sum_{\beta \in \mathcal{A}}\Omega_{\pi, \zeta}(\alpha, \beta) \vs(\beta)=\sum_{\pi_t(\gamma)<\pi_t(\alpha)} \vs(\gamma)-z_{\alpha}\sum_{\pi_b(\gamma)<\pi_b(\alpha)} v^{S}_{\pi^{-1}, \eta}(\gamma)=\delta_{S,S_{0}} (1-z_{\alpha}), 
\end{equation}
where $\delta_{S,S_0}=1$ if and only if $S=S_0$, and is zero otherwise.

We now proceed to the proof of the other assertion in part $3$, that is, $Ker \Omega_{\pi, \zeta} \subset N(\pi, \zeta)$. Let $f: \{0\} \cup \mathcal{A} \to \mathbb{C}$ be a function defined on the upper vertices of the polygon $P(\lambda, \pi, \tau)$ such that $f(0)=0$ and let $\mu(f):=(\mu_{\alpha})_{\alpha} \in \mathbb{C}^\mathcal{A}$ be given by 
\begin{equation}
    \mu_{\alpha}:= f(\alpha^{-1})c_{\pi,\eta}(\alpha^{-1})-f(\alpha)c_{\pi, \eta}(\alpha).
\end{equation}
We remark that any $\mu\in \mathbb{C}^{\mathcal{A}}$ can be obtained via the above construction. Our goal now is to show that if $\mu \in Ker \Omega_{\pi, \zeta}$, then $f$ must be invariant under the action of the $\sigma$-permutation. To this end, we use induction on $k$ to prove the following

\medskip 

\noindent
\textbf{Claim}. For every $\alpha$ such that $\pi_b(\alpha) \leq k$, $f(\sigma(\alpha^{-1}))=f(\alpha^{-1})$.
\medskip

We first prove the base case where $k=1$. Let $\alpha \in \mathcal{A}$ be such that $\pi_b(\alpha)=1$. Then observe that $\sigma^{-1}(0)=\alpha^{-1}$ and 
\begin{equation}
    0=(\Omega_{\pi, \zeta}\mu)_{\alpha} = \sum_{\pi_t(\gamma)<\pi_t(\alpha)} \mu(\gamma)=f(0)-f(\alpha^{-1})c_{\pi, \eta}(\alpha^{-1}).
\end{equation}
So since $f(0)=0$ the above equation implies that $f(\sigma^{-1}(0))=0$. 

We now assume that the claim holds for all $m <k $ and we show that the claim holds for $k$. Assume now that $\alpha$ is such that $\pi_b(\alpha)=k$. Then
\begin{align}
    0=(\Omega_{\pi, \zeta}\mu)_{\alpha} =\sum_{\pi_t(\gamma)<\pi_t(\alpha)} \mu(\gamma)-
       z_{\alpha} \sum_{\pi_b(\gamma)<\pi_b(\alpha)} \Big(c_{\pi^{-1},\eta}(\gamma^{-1})f(\gamma^{-1})-c_{\pi^{-1},\eta}(\gamma)f(\gamma) \Big)\\
    =(f(\sigma(\alpha^{-1}))-f(\alpha^{-1}))c_{\pi,\eta}(\alpha^{-1}),
\end{align}
where the last equality follows from the fact that if $\pi_b(\gamma_2)=1+\pi_b(\gamma_1)<k$ then $f(\gamma_2^{-1})=f(\gamma_1)$ and therefore the intermediate terms in the second sum telescopically cancel each other. Thus, the claim holds and hence there exist $a_S \in \mathbb{C}$ such that $f= \sum_{S\in \Sigma(\pi)\setminus \{S_{0}\}} a_S \one_S$, which implies $w \in N(\pi, \zeta)$.

To prove part $4$, we apply both sides of equation \eqref{eq:inv-big-omega} to $W_{\zeta^{(1)}}$
\begin{equation}
    \mathcal{B}(\lambda, \pi, \zeta) \Omega_{\pi, \zeta}\mathcal{B}^{*}(\lambda, \pi, \zeta)W_{\zeta^{(1)}}=\Omega_{\pi^{(1)}, \zeta^{(1)}}W_{\zeta^{(1)}}.
\end{equation}
Substituting $W_\zeta$ in place of $\mathcal{B}^{*}(\lambda, \pi, \zeta) W_{\zeta^{(1)}}$ will then imply the desired invariance $\mathcal{B}(\lambda, \pi, \zeta) \widetilde{H}(\pi, \zeta)=\widetilde{H}(\pi^{(1)}, \zeta^{(1)})$. The dimension count follows from the rank formula applied to the restriction of $\Omega_{\pi, \zeta}$ to $W_{\zeta}$ and the fact that the kernel of $\Omega_{\pi, \zeta}$ is $(\kappa-1)$-dimensional. 

We now turn to the proof of $\kappa$-dimensionality of $\widetilde{N}(\pi, \zeta)$ in part $5$. Recall that $\widetilde{N}(\pi, \zeta)$ is spanned by the vectors $\vs$ corresponding to the cycles $S$ of the cycle decomposition of $\sigma$. Assume that there is a linear relation between $\vs$, that is,
\begin{equation} \label{eq: linear independence}
    \sum_{S \in \Sigma} c_S \vs=0.
\end{equation}
Let $\alpha,\beta$ be two symbols whose corresponding intervals are adjacent in the top permutation. In other words, $\pi_t(\alpha)=\pi_t(\beta)+1$. Comparing the $\alpha$-th entries in both sides of \eqref{eq: linear independence}, we get that 
\begin{equation}
    -c_{S(\alpha)}c_{\pi, \eta}(\alpha)+c_{S(\beta)}c_{\pi,\eta}(\beta)=0,
\end{equation}
where, for any symbol $\gamma \in \mathcal{A}$, $S(\gamma) \in \Sigma$ is the cycle passing through the right-end vertex of the interval corresponding to $\pi_t(\gamma)$. Therefore for each $\alpha \in \mathcal
A$ we get
\begin{equation}
    c_{S(\beta)}w_{\eta}(\alpha)=c_{S(\alpha)}.
\end{equation}
This implies that either all $c_S$'s are zero or for every $i \in \{0,1,\ldots,d\}$ 
\begin{equation}
    \prod_{i<\pi_t(\beta)\leq\sigma(i)} w_{\eta}(\beta)=1,
\end{equation}
which in turn implies that $-\Omega\eta=\zeta \in \mathbb{Z}^{d}$ contrary to our standing assumption. We conclude the proof by showing orthogonality of $\widetilde{H}(\pi,\zeta)$ and $\widetilde{N}(\pi, \zeta)$. For this, we first show that $\widetilde{N}(\pi, \zeta) \subset W_{\zeta}$. Let $f:\{0\} \cup \mathcal{A} \to \mathbb{C}$ be a function invariant under $\sigma$ and $\mu=\mu(f)$. 
\begin{equation}
    \langle \mu, s_\zeta \rangle= \sum_{\gamma} \mu_{\gamma}-\sum_{\gamma}\Big(c_{\pi^{-1},\eta}(\gamma^{-1})f(\gamma^{-1})-c_{\pi^{-1},\eta}(\gamma)f(\gamma)  \Big)=0,
\end{equation}
as $f$ is invariant under $\sigma.$  Note that an element in $\tilde{N}(\pi, \zeta)$ is of the form $\mu=\mu(f)$ for some $\sigma-$invariant $f$. Thus, for any $\mu \in \widetilde{N}(\pi, \zeta)$ and any $g \in W_{\zeta}$
\begin{equation}
    \langle \mu, \Omega_{\pi, \zeta} g \rangle= g^{*}\Omega_{\pi, \zeta}^{*} \mu= -g^{*} \Omega_{\pi, \zeta} \mu=0.
\end{equation}
This finishes the proof of the Lemma. 
\end{proof}

\noindent\textbf{Real Cocycle.} 
Let $\iota_\mathbb{R}: \mathbb{C}^{d} \to \mathbb{C}^{2d}$ be the map sending $f$ to $(f, \overline{f})$. We denote the image $\iota_{\mathbb{R}}(\mathbb{C}^d)$ of $\mathbb{C}^d$ under $\iota_{\mathbb{R}}$ by $\real$, and remark that it is a $2d$-dimensional real vector space. The action of $\mathbb{C}$ on $\mathbb{C}^{2d}$ naturally endows $\real$ with a scalar action by $\mathbb{C}$ via $c.(f,\overline{f}):= (cf, \overline{c}\overline{f})$. Henceforth, for every complex subspace $V \subset \mathbb{C}^d$ we call $V_{\mathbb{R}}:= \iota_\mathbb{R}(V)$ the \textit{real subspace} associated to $V$. For every $(\lambda,\pi,\zeta)$, $\mathcal{B}(\lambda, \pi, \zeta)$ induces the linear map $\mathcal{B}_{\mathbb{R}}(\lambda,\pi,\zeta):\real\to \real$ given by 

\begin{equation}\label{eq:real-cocycle}
\mathcal{B}_{\mathbb{R}}(\lambda,\pi,\zeta)(f,\overline{f}):=\big(\mathcal{B}(\lambda,\pi,\zeta)f, \mathcal{B}(\lambda,\pi,-\zeta)\overline{f}\big).
\end{equation}
In other words, the following diagram commutes
  \begin{equation}\label{eq:commute-diagram}
      \begin{tikzcd}
\mathbb{C}^d \arrow[r, "\mathcal{B}"] \arrow[d, "\iota_{\mathbb{R}}"'] & \mathbb{C}^d \arrow[d, "\iota_{\mathbb{R}}"] \\
\real \arrow[r, "\mathcal{B}_{\mathbb{R}}"'] & \real
\end{tikzcd}
  \end{equation}

In this representation, $(\mathcal{B}_{\mathbb{R}}^{t})^{-1}=\iota_\mathbb{R} \circ (\mathcal{B}^{*})^{-1}$. $\mathcal{B}_{\mathbb{R}}(\lambda, \pi, \zeta)$'s induce a linear cocycle on $\Delta \times \mathfrak{R} \times \mathbb{T}^d\times \real$, over $\mathcal{R}_{\mathbb{T}^d}$.

Let $H, H_{\mathbb{R}}$ be the fiber bundles over $\Delta \times \mathfrak{R}\times \mathbb{T}^d$
whose fibers over the point $(\lambda,\pi,\zeta)$ are $\widetilde{H}(\pi, \zeta)/\mathbb{C}s_\zeta$ and $\widetilde{H}_{\mathbb{R}}(\pi,\zeta)/\mathbb{C}(s_\zeta,\overline{s_\zeta})$, respectively.


\begin{prop}\label{prop:symplectic-form}
      $\Omega_{\pi, \zeta}$'s induce a $\mathcal{B}$-invariant family of (continuously varying with $\zeta \in \mathbb{T}^d\setminus \{0\}$) Hermitian forms $\omega_{\pi,\zeta}:\widetilde{H}(\pi,\zeta)\times \widetilde{H}(\pi,\zeta)\to \mathbb{C}$ by 
\begin{equation}\label{eq:inv-small-omega}
\omega_{\pi,\zeta}\big(\Omega_{\pi,\zeta}f,\Omega_{\pi,\zeta}g\big):=\frac{i}{2}\langle \Omega_{\pi,\zeta}f,g\rangle,
 \end{equation}
 where $f,g \in W_\zeta$. The above family induces a (continuously varying with $\zeta \in \mathbb{T}^d\setminus \{0\}$) family of $\mathcal{B}_{\mathbb{R}}$-invariant symplectic (real)-bilinear forms $\omega^{\mathbb{R}}_{\pi, \zeta}: \widetilde{H}_{\mathbb{R}}(\pi, \zeta) \times \widetilde{H}_{\mathbb{R}}(\pi, \zeta) \to \mathbb{R}$ in the following manner

\begin{equation}
\omega^{ \mathbb{R}}_{\pi, \zeta} ((\Omega_{\pi, \zeta} f, \Omega_{\pi, -\zeta}\overline{f}), (\Omega_{\pi, \zeta} g, \Omega_{\pi, -\zeta}\overline{g})):= -\mathrm{Im}\big( \omega_{\pi,\zeta}(\Omega_{\pi, \zeta}f, \Omega_{\pi,\zeta}g)).
\end{equation}
Moreover, this family restricted to $H_{\mathbb{R}}$
is non-degenerate. 

\end{prop}

\begin{proof}
We note that continuity is clear from the definition. Let us now first verify the conjugate transposition property, that is, we show 
\begin{equation}
    \omega_{\pi, \zeta}(\Omega_{\pi, \zeta}f, \Omega_{\pi, \zeta} g)=\overline{\omega_{\pi, \zeta}(\Omega_{\pi, \zeta}g, \Omega_{\pi, \zeta} f)}.
\end{equation} 
To this end, we use equation \eqref{Omega+Omega*} to see that 
\begin{equation}
    \overline{\langle \Omega_{\pi, \zeta}f,g\rangle}+\langle \Omega_{\pi, \zeta}g,f\rangle=(g^* \Omega_{\pi, \zeta}f)^{*}+f^*\Omega_{\pi, \zeta}g=f^*(\Omega_{\pi, \zeta}+\Omega^{*}_{\pi, \zeta})g=0,
\end{equation}
as $f,g \in W_{\zeta}= \langle s_{\zeta}\rangle ^{\perp}$.
It follows from Proposition \ref{inv-Omega} that for every $f,g \in W_{\zeta^{(1)}}$
\begin{align}
\omega_{\pi^{(1)}, \zeta^{(1)}}(\mathcal{B}\Omega_{\pi, \zeta}f, \mathcal{B} \Omega_{\pi, \zeta} g)&=\omega_{\pi^{(1)},\zeta^{(1)}}\big(\Omega_{\pi^{(1)},\zeta^{(1)}}\mathcal({B}^*)^{-1}f,\Omega_{\pi^{(1)},\zeta^{(1)}}(\mathcal{B}^*)^{-1}g\big)\\
&=\frac{i}{2}\langle \Omega_{\pi^{(1)},\zeta^{(1)}}(\mathcal{B}^*)^{-1}f, (\mathcal{B}^*)^{-1}g\rangle  =  
\frac{i}{2}\langle \mathcal{B}\Omega_{\pi,\zeta}f, (\mathcal{B}^*)^{-1}g\rangle \\ &= \frac{i}{2} 
\langle \Omega_{\pi,\zeta}f ,g \rangle = \omega_{\pi,\zeta}(\Omega_{\pi,\zeta}f,\Omega_{\pi, \zeta}g) .  
\end{align}
Regarding $\omega^{\mathbb{R}}_{\pi,\zeta}$, bi-linearity is obvious. In order to show that  $\omega^{ \mathbb{R}}_{\pi, \zeta}$ is symplectic, we need to prove that for every $f\in W_\zeta$, $\langle \Omega_{\pi,\zeta} f,f\rangle$ is a purely imaginary number. This follows from the conjugate transposition property of $\omega_{\pi, \zeta}$.
To prove non-degeneracy, assume that $f\in W_\zeta$ is such that \begin{equation}
    \forall g\in W_\zeta, \;\mathrm{Re}(\langle \Omega_{\pi,\zeta} f,g\rangle)=0.
\end{equation} 
Now, replacing $g$ with $cg$ where $c\in \mathbb{C}$ is arbitrary, we deduce that $\mathrm{Re}(\overline{c}\langle \Omega_{\pi,\zeta} f,g\rangle)=0$. This implies that for every $g \in W_\zeta$, $\langle \Omega_{\pi,\zeta} f,g\rangle=0$ and therefore $\Omega_{\pi,\zeta}f\in \langle s_\zeta\rangle$. The invariance of $\omega^{\mathbb{R}}_{\pi,\zeta}$ under $\mathcal{B}_{\mathbb{R}}$ follows from the commutativity of diagram \eqref{eq:commute-diagram} and the invariance of $\omega_{\pi, \zeta}$ under $\mathcal{B}$..
\end{proof}

From the definition of the real cocycle $\mathcal{B}_{\mathbb{R}}$, it can be readily verified that for every $v \in \mathbb{C}^d$ and every $(\lambda, \pi, \zeta)$
\begin{equation}
     \frac{ \|\mathcal{B}_{n}(\lambda, \pi, \zeta).v\|}{\|v\|}=\frac{\|(\mathcal{B}_{\mathbb{R}})_n(\lambda, \pi, \zeta). (v, \overline{v})\|}{\|(v,\overline{v})\|}.
\end{equation}
The second part of Proposition \ref{prop:symplectic-form} regarding the invariance of a family of symplectic forms under $\mathcal{B}_{\mathbb{R}}$ together with the above observation yield the following 

\begin{cor}\label{cor:sym-spec}
    The spectrum of the restriction of the cocycle $\mathcal{B}_{\mathbb{R}}$ to the fiber bundle $H_{\mathbb{R}}$ is symmetric with respect to zero. Every Lyapunov exponent $\chi$ of the cocycle $\mathcal{B}$ restricted to $H$ is also a Lyapunov exponent of $\mathcal{B}_{\mathbb{R}}$ restricted to $H_{\mathbb{R}}$ with twice the multiplicity. Moreover, the Oseledets subbundles of $\mathcal{B}_{\mathbb{R}}$ (restricted to $H_{\mathbb{R}}$) are obtained via applying the map $\iota_{\mathbb{R}}$ to the Oseledets subbundles of $\mathcal{B}$.
\end{cor}

    


\section{Proof of Main Theorem}
Here we use the tools developed in the previous section to prove our main theorem. 
\begin{proof}[Proof of Main Theorem]
The symmetry of the spectrum follows from Corollary \ref{cor:sym-spec}, the existence of the family of continuous invariant symplectic forms, and a standard application of Poincare recurrence to compact subsets of $\mathbb{T}^d \setminus \{0\}$. In what follows we will show that $(\mathcal{B}^*)^{-1}$ possesses at least $\kappa+1$ zero exponents. 

    Fix a permutation $\pi$ in the Rauzy class $\mathfrak{R}$ and let $S_0, S_1, \cdots, S_{\kappa-1}$ be an ordering of the set of cycles of the decomposition of $\pi$ into cycles. We will proceed by conjugating the cocycle $(\mathcal{B}^{*})^{-1}$ to one of a special form which will allow us to conclude the existence of $\kappa+1$ zero exponents. To this end, we define a family of matrices $C_\zeta$ depending continuously on $\zeta \in \mathbb{T}_{\pi}^{2g}\setminus\{0\}$, such that
    \begin{equation}\label{eq:Upper block triangular form}
      C_{\zeta^{(1)}}^{-1}(\mathcal{B(\lambda, \pi, \zeta})^{*})^{-1}C_\zeta =   \begin{pNiceArray}{c|c}
\begin{matrix}
1 & & & \\
& u_1 & & \\
& & \ddots & \\
& & & u_{\kappa-1}
\end{matrix}  & \Block{1-1}{\mathlarger{*}} \\
\hline
\Block{1-1}{0} &  \begin{matrix}
\psi & 0 & \cdots & 0 \\
\hline
* & * & \cdots & * \\
\vdots & \vdots & \ddots & \vdots \\
* & * & \cdots & *
\end{matrix}
\end{pNiceArray},
    \end{equation}
where $u_i=u_i(\pi, \lambda, \zeta)$'s have norm $1$ and 
\begin{equation}
    \psi(\lambda, \pi, \zeta)= \frac{\|s_{\zeta}\|^2}{\|s_{\zeta^{(1)}}\|^2}.
\end{equation} 

    To define $C_\zeta$, for every $\zeta \in \mathbb{T}^d\setminus\{0\}$ we provide a basis for the fiber corresponding to $\zeta$.  Let now $\eta \in \mathbb{R}^\mathcal{A} \cap H(\pi)$ be the unique element such that $-\Omega_{\pi} \eta=\zeta$. For every $0\leq i \leq \kappa-1$, let $v_{i}(\zeta):=v^{S_{i}}_{\pi, \eta}$ (see \eqref{eq:def-v_S} for the definition of $\vs$). We now remark that $v_{0}(\zeta), v_{1}(\zeta), \ldots, v_{\kappa-1} (\zeta)$ form a basis of $\widetilde{N}(\pi, \zeta)$ (defined in Lemma \ref{lem:properties-of-omega}). The basis $\big\{v_{0}(\zeta), v_{1}(\zeta), \ldots, v_{\kappa-1} (\zeta), s_\zeta \big\} $ provides a trivialization of the bundle $\widetilde{N}(\pi, \zeta)\oplus \mathbb{C}s_\zeta$ over $\mathbb{T}^{d} \setminus \{0\}$ and therefore, by general principles, it is possible to extend this basis to a continuously varying basis $\big\{ v_{0}(\zeta), v_{1}(\zeta), \ldots, v_{\kappa-1} (\zeta), s_\zeta, w_1, \ldots, w_{2g-2}\big\}$ for the whole fiber over $\zeta$ such that $\big\{ s_\zeta, w_1, \ldots, w_{2g-2}\big\}$ is an orthonormal basis for $\widetilde{H}(\pi, \zeta)$. 
    We take $C_\zeta$ to be the matrix whose columns are respectively, $v_{0}(\zeta), v_{1}(\zeta), \ldots, v_{\kappa-1} (\zeta), s_\zeta, w_1, \ldots, w_{2g-2}$. The equality of the first $\kappa$ columns in \eqref{eq:Upper block triangular form} follows from Lemma \ref{lem:inv-v_S} and Proposition \ref{prop:N-span}. For the $(\kappa+1)$st column, since $s_{\zeta}$ is perpendicular to other elements of the basis, 
    \begin{equation}
       \psi=\frac{ \langle (\mathcal{B}^{*})^{-1} s_\zeta , s_{\zeta^{(1)}}\rangle }{\|s_{\zeta^{(1)}}\|^2}=\frac{ \langle  s_\zeta , \mathcal{B}^{-1}s_{\zeta^{(1)}}\rangle }{\|s_{\zeta^{(1)}}\|^2}= \frac{ \langle  s_\zeta , s_\zeta\rangle }{\|s_{\zeta^{(1)}}\|^2}.
    \end{equation}
    The presence of zero entries in the $(\kappa+1)$st row follows from the invariance of $W_\zeta$ under $\mathcal{B}^*$, see \eqref{eq:inv-W}.


Now, since $\|C_{\zeta}^{\pm 1}\|_{\mathrm{HS}}^2$ restricted to $\mathbb{T}^d$ are rational functions of $z_1,\ldots,z_n$, by general principles (see for instance \cite{Aroundunitcircle}), they are $\log$-integrable on $\mathbb{T}^d$ (with respect to $m_{\mathbb{T}^d}$). So, we deduce that the spectrum of the cocycles $C_{\zeta^{(1)}}^{-1}(\mathcal{B}^*)^{-1}C^{}_\zeta$  and $(\mathcal{B}^*)^{-1}$ are the same.   

Then, the main result of \cite{key1988lyapunov} asserts that the spectrum of block triangular cocycles such as $C_{\zeta^{(1)}}^{-1}(\mathcal{B}^*)^{-1}C^{}_\zeta$ is the union of the spectra of the diagonal blocks. The upper left corner block contributes $\kappa$ zero exponents to the spectrum. The bottom right corner is itself a lower block triangular cocycle decomposing into a one-dimensional block and a $(2g-2)$-dimensional block. The contribution from the one-dimensional block to the spectrum is the growth rate of Birkhoff averages of $\log \psi$, which is zero as

\begin{equation}
\int_{\Delta \times \mathbb{T}^d}\log\psi=\int_{\Delta \times \mathbb{T}^d}\log\|s_{\zeta}\|^2-\int_{\Delta \times \mathbb{T}^d}\log\|s_{\zeta^{(1)}}\|^2=0.    
\end{equation}
 
The integrability of $\log\|s_{\zeta}\|^2$ is a result of $\|s_{\zeta}\|_{\mathrm{HS}}^2$ restricted to $\mathbb{T}^d$ being a polynomial of $z_1,\ldots,z_n$. This establishes the existence of $\kappa+1$ zero exponents.
\end{proof}


To finish the study of the degeneracy of and nondegeneracy of the cocycle in the $g=1$ case, we must mention that the main theorem of the previous work of the authors yields the following 

\begin{thm} (\cite{Rajabzadeh-Safaee})
Let $d>1$ be an integer and $\pi=(\pi_t, \pi_b)$ be a rotation-type permutation. Then, the top Lyapunov exponent of the twisted cocycle with respect to the invariant measure $\mu \times m_{\mathbb{T}^d}$ is positive. 
\end{thm}

\section{Appendix}
We give explicit examples of self-similar IETs (of genus $1$ and $2$) together with their (twisted) invariant and covariant sections and discuss the application to substitution dynamics. In the examples below, for any $\zeta \in H(\pi)$, we take the unique $\eta \in H(\pi)$ so that $-\Omega_\pi \eta=\zeta$ and let $v^S_{\pi,\zeta}:=v^S_{\pi,\eta}.$  Recall that for $\zeta=(\zeta_\alpha)_{\alpha\in \mathcal{A}} \in \mathbb{R}^\mathcal{A}$, we let $z_\alpha=\exp(2\pi i \zeta_\alpha)$.
\begin{example}
Let $\pi=\left(\begin{smallmatrix}
    ABC\\
    CBA
\end{smallmatrix}\right)$ and consider the Rauzy loop $\gamma$ produced by the following renormalization steps
\begin{equation}
    C>A \rightarrow B>C \rightarrow C>B \rightarrow A>C \rightarrow B>A \rightarrow A>B,
\end{equation} 
where by $C>A$ we mean that the first renormalization operation is of top type. Then the Rauzy renormalization matrix corresponding to $\gamma$ is 
\begin{equation}
B_{\gamma}=
\begin{pmatrix}
    1 & 2 & 2\\
    1 & 4 & 3\\
    1 & 1 & 2
\end{pmatrix}.
\end{equation}
Then the left eigenvector $\lambda:=(3-2\sqrt{2},\sqrt{2}-1,\sqrt{2}-1)$ of $B_\gamma$ yields a length vector corresponding to a self-similar IET (periodic point of order $6$ for $\mathcal{R}_{R}$) whose renormalization combinatorics consists of successively going around the loop $\gamma$. The corresponding twisted cocycle $F:\mathbb{T}^3 \times \mathbb{C}^3 \to \mathbb{T}^3 \times \mathbb{C}^3$ (to the power $6$) is given by
\begin{equation}
    F(\zeta, v):= (B_\gamma \zeta, \mathcal{B}_{\gamma}(\zeta)v)),
\end{equation}
where  
\begin{equation}
    \mathcal{B}_{\gamma}(\zeta)=\mathcal{B}_{6}(\lambda, \pi, \zeta):= \begin{pmatrix}
        1 & z_az_c+z_az_bz_c & z_a+z_az_b^2z_c \\
        1 & z_az_c+z_az_bz_c+z_az_b^2z_c^2+z_az_b^3z_c^2 & z_a+z_az_b^2c+z_az_b^4z_c^2\\
        1 & z_az_c & z_a+z_az_bz_c
    \end{pmatrix}.
\end{equation}
The action of $B_\gamma$ on $\mathbb{T}^{3}$ is multiplicative, that is, 
\begin{equation}
    B_{\gamma}.\zeta= (z_az_b^2z_c^2, z_az_b^4z_c^3, z_az_bz_c^2).
\end{equation}
We now show by explicit calculation that $s_\zeta$ is an invariant section for the twisted cocycle. 
\begin{equation}
  \mathcal{B}_\gamma(\zeta)s_\zeta=\mathcal{B}_\gamma(\zeta)
\begin{pmatrix}
1-z_a \\
1-z_b \\
1-z_c
\end{pmatrix}
=
\begin{pmatrix}
1-z_a z_b^2 z_c^2  \\
1-z_a z_b^4 z_c^3 \\
1-z_a z_b z_c^2 
\end{pmatrix}=s_{B_\gamma\zeta}.
\end{equation}
In this example, the function $g$ defined by
\[
    g(z)=z_az_cz_b^{-1}
\]
is invariant under the action of $B_\gamma$. Henceforth, we concentrate on the action of the twisted cocycle restricted to the sub-torus defined by $\{g(z)=1\}$ and exhibit the covariant sections explicitly (see Section \ref{sec:inv-structure}).
\begin{align}
   \mathcal{B}_\gamma(\zeta)^*v^{S_1}_{\pi,B_{\gamma}\zeta}&= B_\gamma(\zeta)^*
\begin{pmatrix}
    1 \\
    -1 \\
    z_a^{-1}z_b^{-2}z_c^{-2}
\end{pmatrix}= \begin{pmatrix}
    z_a^{-1} z_b^{-2} z_c^{-2} \\ z_a^{-2} z_b^{-2} z_c^{-3} - z_a^{-1} z_b^{-3} z_c^{-2} - z_a^{-1} z_b^{-2} z_c^{-2} \\ z_a^{-2} z_b^{-3} z_c^{-3} - z_a^{-1} z_b^{-4} z_c^{-2} + z_a^{-2} z_b^{-2} z_c^{-2}
\end{pmatrix}\\
&=z_a^{-1}z_b^{-2}z_c^{-2}\begin{pmatrix}
    1 \\ z_a^{-1}  z_c^{-1} -  z_b^{-1} - 1 \\ z_a^{-1} z_b^{-1} z_c^{-1} -  z_b^{-2}  + z_a^{-1} 
\end{pmatrix}=z_a^{-1}z_b^{-2}z_c^{-2}\begin{pmatrix}
    1 \\ - 1 \\   z_a^{-1} 
\end{pmatrix}=(z_a^{-1}z_b^{-2}z_c^{-2})v^{S_1}_{\pi,\zeta},
\end{align}
shows that $v^{S_1}_{\pi, \zeta}$ is a covariant section up to multiplication by a constant of unit norm (see Lemma \ref{lem:inv-v_S} and Proposition \ref{prop:N-span}). Below, we will demonstrate that the distinguished section $v_{\pi,\zeta}$ is covariant. 
\begin{align}
   \mathcal{B}_\gamma(\zeta)^*v_{\pi, \mathcal{B}_\gamma\zeta}&= \mathcal{B}_\gamma(\zeta)^*
\begin{pmatrix}
    -1 \\
    z_az_bz_c^2 \\
    -z_az_bz_c^2
\end{pmatrix}= \begin{pmatrix}
   -1 \\ 
   -z_a^{-1}z_c^{-1}-z_a^{-1}z_b^{-1}z_c^{-1}+z_b^{-2}+z_b^{-1}+z_c
   \\ 
   z_b^{-1}z_c-z_a^{-1}z_b^{-2}z_c^{-1}-z_a^{-1}+z_b^{-3}-z_c
\end{pmatrix}
=\begin{pmatrix}
    -1 \\z_c \\  -z_c
\end{pmatrix}=v_\zeta,
\end{align}
where in all verifications we have used the fact that $z_az_c=z_b$.





\end{example}

\begin{example}\label{ex:4IET} (Rauzy loop for 4-IETs) Take the loop $\gamma$ given by
\begin{equation}
    D>A, D>B, C>D, D>C, A>D, A>C, B>A, A>B,
\end{equation}
where the starting permutation is 
$
    \pi= \left(\begin{smallmatrix}
     A & B & C & D\\
     D & C & B & A
    \end{smallmatrix}\right).$
The corresponding substitution is then 
\begin{align}
    &A \to ADBD\\
    &B \to ADBDBD\\
    &C \to ADCCD\\
    &D \to ADCD
\end{align}
and the corresponding (untwisted) renormalization matrix is 
\begin{equation}
    B_\gamma=\begin{pmatrix}
     1 & 1 & 0 & 2\\
     1 & 2 & 0 & 3\\
     1 & 0 & 2 & 2\\
     1 & 0 & 1 & 2
    \end{pmatrix}.
\end{equation}
As in the previous example, we write the action of $B_\gamma$ on the corresponding torus (in this case $\mathbb{T}^4=\mathbb{R}^4/\mathbb{Z}^4$) multiplicatively.
\begin{equation}
    B_\gamma. (z_a, z_b, z_c, z_d)= (z_az_bz_d^2, z_az_b^2z_d^3, z_a z_c^2z_d^2, z_az_cz_d^2).
\end{equation}
The IET whose renormalization combinatorics follows $\gamma$ infinitely many times is a self-similar IET whose length vector $\lambda$ is the left eigenvector with the largest eigenvalue of $B_\gamma$. 
The corresponding twisted cocycle is
\begin{equation}
    \mathcal{B}_\gamma(\zeta)= \begin{pmatrix}
     1 & z_az_d       &  0  & z_a+z_az_bz_d\\
     1 & z_az_d+z_az_bz_d^2 &  0  & z_a+z_az_bz_d+z_az_b^2z_d^2\\
     1 & 0 & z_az_d+z_az_cz_d & z_a+z_az_c^2z_d\\
     1 & 0 & z_az_d & z_a+z_az_cz_d
    \end{pmatrix}.
\end{equation}
 We check again explicitly the equations for the invariance of $s_\zeta$ and the covariance of $v_{\pi,\zeta}$ under renormalization. 

\begin{align}
    \mathcal{B}_\gamma(\zeta)s_\zeta&=\mathcal{B}_\gamma(\zeta)\begin{pmatrix}
        1-z_a\\
        1-z_b\\
        1-z_c\\
        1-z_d
    \end{pmatrix}=\begin{pmatrix}
        1-z_az_bz_d^2\\
        1-z_az_b^2z_d^3\\
        1- z_a z_c^2z_d^2\\
        1-z_az_cz_d^2
    \end{pmatrix}=s_{{B}_\gamma(\zeta)},\\
\mathcal{B}_\gamma(\zeta)^{*}v_{B_\gamma\zeta}&= 
    \mathcal{B}_{\gamma}(\zeta)^{*}
    \begin{pmatrix}
    z_az_b^2z_c^{-1}z_d^3-1 \\ z_bz_d-z_az_b^2z_c^{-1}z_d^3 \\ z_az_cz_d^2-z_bz_d \\z_bz_c^{-1}z_d-z_az_cz_d^2
    \end{pmatrix}=
\begin{pmatrix}
    z_bz_c^{-1}z_d-1\\
    z_a^{-1}z_b-z_bz_c^{-1}z_d\\
    z_d-z_bz_c^{-1}z_d\\
    z_a^{-1}z_bz_c^{-1}z_d-z_d\\
\end{pmatrix}=v_\zeta.
\end{align}
\end{example}

\medskip

\noindent
\textbf{Application to substitution dynamics.} Motivated by discussions with Boris Solomyak, we apply here our result about the existence of the invariant section in the more general setting of substitutions. The main finding in this subsection is Proposition \ref{prop:Singularity_of_spectrum} which detours lengthy calculations in previous works \cite{Baake, Baake_binary_inflation_rules} to establish the singularity of the spectrum for a large class of substitutions. This proposition appears, among other things, in a somewhat different form in the upcoming work of Solomyak \cite{solomyak_twisted_cocycle} on the Lyapunov spectrum of the twisted cocycle for substitutions\footnote{The proof of this proposition was communicated by the second-named author to Boris Solomyak on 4 Sept 2024.}. 



Let $\mathcal{A}=\{0,1\}$, and $\mathcal{A}^{+}:= \bigcup_{n \in \mathbb{N}} \mathcal{A}^{n}$ be the set of all finite words of the alphabet $\mathcal{A}$. Let $\vartheta: \mathcal{A} \to \mathcal{A}^{+}$ be a substitution. It is well-known that $\vartheta$ induces a 
dynamical system on a zero-entropy invariant subset of the left shift map $\sigma:\mathcal{A}^{\mathbb{N}} \to \mathcal{A}^{\mathbb{N}}$. This dynamical system is often studied via the corresponding renormalization dynamics. The renormalization dynamics is given by a $2$ by $2$ matrix $S_{\vartheta}=\left(\begin{smallmatrix}
    a & b\\
    c & d
\end{smallmatrix}\right)$,
where $a,c$ ($b,d$) are the number of $0$s ($1$s) that appear in $\vartheta(0), \vartheta(1)$, respectively. The definition of the spectral (twisted) cocycle denoted by $\mathcal{B}(\zeta)$ where $\zeta:=(z_1, z_2) \in \mathbb{S}^{!} \times \mathbb{S}^1 \subset \mathbb{C} \times \mathbb{C}$, is then analogous to the one given in Example \ref{ex:4IET} (see \cite{Bufetov-Solomyak-MathZ}). For a recent survey on 
substitutions and their spectral properties, see \cite{BufetovSolomyak2023}.

\begin{prop} \label{prop:Singularity_of_spectrum} Let $\vartheta$ be such that one of the following holds:

\begin{itemize}
    \item Either the characteristic polynomial of $S_\vartheta$
    is irreducible over $\mathbb{Q}$ and the top eigenvalue $\lambda$ of $S_\vartheta$ satisfies \begin{equation}\label{eq:lambda^n>2min}
        \lambda^n > 2\min \{a_n+c_n, b_n+d_n\}
    \end{equation} 
    for some $n \in \mathbb{N}$, where 
     \[S_\vartheta^{n}=\begin{pmatrix}
        a_n & b_n\\
        c_n & d_n
    \end{pmatrix},\]
    \item or $\vartheta$ is a substitution of constant length $q$.
\end{itemize}
    Then the substitution system corresponding to $\vartheta$ has pure singular spectrum.
\end{prop}
\begin{proof}
    We proceed by showing that the top exponent $\chi^{+}$ of $\tw$ with respect to the Lebesgue measure $m_{\mathbb{T}^2}$ satisfies the following inequality in either case
    \begin{equation}
        \chi^{+} <\frac{1}{2}\log(\lambda).
    \end{equation}
    The singularity of the spectrum then follows from \cite[Theorem 2.4]{Bufetov-Solomyak-MathZ}.
    Let $s_\zeta$ be as in \eqref{eq:def-inv-bundle}. Notice that the invariance of $s_\zeta$ under the twisted cocycle $\mathcal{B}$ implies that one of the Lyapunov exponents of the twisted cocycle $\mathcal{B}$ (with respect to the Lebesgue measure on $\mathbb{T}^2$) is zero and the other is given by 
    \begin{equation}
        \chi^{+}=\int_{\mathbb{T}^2} \log |\det \mathcal{B}(\zeta)|\, {d}m_{\mathbb{T}^2}(\zeta).
    \end{equation}
    Without loss of generality, we may now assume that equation \eqref{eq:lambda^n>2min} is satisfied for $n=1$. Denoting the twisted matrix by\[ 
        \tw(\zeta):=\begin{pmatrix}
            a(z_0,z_1) & b(z_0, z_1)\\
            c(z_0,z_1) & d(z_0, z_1)
 \end{pmatrix},\]
     we calculate the determinant by taking a basis consisting of $e_1, s_\zeta$ and we get 
    \begin{equation}
          \det \tw(z_0, z_1)=\frac{\det 
          \begin{pmatrix}
              a(z_0,z_1) & 1-z^{a}_0z^{b}_1\\
              c(z_0, z_1) & 1-z^{c}_{0}z^{d}_1
          \end{pmatrix}}{\det 
          \begin{pmatrix}
              1 & 1-z_0\\
             0 & 1-z_1
          \end{pmatrix}}= \frac{1}{1-z_1}\det 
          \begin{pmatrix}
              a(z_0,z_1) & 1-z^{a}_0z^{b}_1\\
              c(z_0, z_1) & 1-z^{c}_{0}z^{d}_1
          \end{pmatrix},
    \end{equation}
     where we used the fact that
     \[ \tw(\zeta)s_\zeta= s_{S_\vartheta \zeta}= \begin{pmatrix}
          1-z^{a}_0z^{b}_1\\
           1-z^{c}_{0}z^{d}_1
     \end{pmatrix}.\]
     The logarithmic Mahler measure\footnote{For a nonzero polynomial $P \in \mathbb{C}[z_1,\dots, z_n]$, the logarithmic Mahler measure of $P$ is given by the following well-defined integral
     $m(P)=\int_{\mathbb{T}^d} \log |P(e^{2\pi i x_1}, \dots, e^{2\pi x_n})|\, {d}m_{\mathbb{T}^d}(x_1, \dots, x_n)$. For basic properties of the Mahler measure see the book \cite{Aroundunitcircle}.
     } of $1-z_1$ being zero, we get that $\chi^{+}=m(P)$, where 
    \begin{equation}
     P(z_0,z_1)= a(z_0,z_1)+z^{a}_0z^{b}_1c(z_0,z_1) - c(z_0,z_1)- z^{c}_{0}z^{d}_1a(z_0, z_1).
    \end{equation}
    By the well-known $L^2$ inequality for the logarithmic Mahler measure, we have
    \begin{equation}
        \chi^{+}=m(P) \leq \frac{1}{2} \log (2(a+c)).
    \end{equation}
    Applying the same idea to the vector $e_2$ in place of $e_1$ we get
    \begin{equation}
        \chi^{+} \leq \frac{1}{2}\log \min \{2(a+c), 2(b+d)\}<\frac{1}{2}\log(\lambda).
    \end{equation}
     In the  constant length-$q$ case, the calculation of the exponent boils down to the calculation of the Mahler measure of $a(z,z)-c(z,z)$
    which is a polynomial with coefficients in $\{-1,0,1\}$ and of degree at most $q$. It is then easy to see that by Gonçalves inequality (see \cite{Aroundunitcircle}) 
    \begin{equation}
       \chi^+=m(P) < \frac{1}{2}\log(q)<\frac{1}{2}\log(\lambda),
    \end{equation}
    finishing the proof. 
\end{proof}


\begin{thebibliography}{BMMS24}

\bibitem[AD16]{AvilaDelecroixweak}
Artur Avila and Vincent Delecroix.
\newblock Weak mixing directions in non-arithmetic {V}eech surfaces.
\newblock {\em J. Amer. Math. Soc.}, 29(4):1167--1208, 2016.
\newblock \href {https://doi.org/10.1090/jams/856} {\path{doi:10.1090/jams/856}}.

\bibitem[AF07]{AvilaForniweak}
Artur Avila and Giovanni Forni.
\newblock Weak mixing for interval exchange transformations and translation flows.
\newblock {\em Ann. of Math. (2)}, 165(2):637--664, 2007.
\newblock \href {https://doi.org/10.4007/annals.2007.165.637} {\path{doi:10.4007/annals.2007.165.637}}.

\bibitem[AF08]{AthreyaForni}
Jayadev~S. Athreya and Giovanni Forni.
\newblock Deviation of ergodic averages for rational polygonal billiards.
\newblock {\em Duke Math. J.}, 144(2):285--319, 2008.
\newblock \href {https://doi.org/10.1215/00127094-2008-037} {\path{doi:10.1215/00127094-2008-037}}.

\bibitem[AFS23]{AvilaForniSafaee}
Artur Avila, Giovanni Forni, and Pedram Safaee.
\newblock Quantitative weak mixing for interval exchange transformations.
\newblock {\em Geom. Funct. Anal.}, 33(1):1--56, 2023.
\newblock \href {https://doi.org/10.1007/s00039-023-00625-y} {\path{doi:10.1007/s00039-023-00625-y}}.

\bibitem[AHCF24]{WeakmixingBilliard}
Francisco Arana-Herrera, Jon Chaika, and Giovanni Forni.
\newblock Weak mixing in rational billiards, 2024.
\newblock URL: \url{https://arxiv.org/abs/2410.11117}, \href {https://arxiv.org/abs/2410.11117} {\path{arXiv:2410.11117}}.

\bibitem[AL18]{AvilaLeguil}
Artur Avila and Martin Leguil.
\newblock Weak mixing properties of interval exchange transformations \& translation flows.
\newblock {\em Bull. Soc. Math. France}, 146(2):391--426, 2018.
\newblock \href {https://doi.org/10.24033/bsmf.2761} {\path{doi:10.24033/bsmf.2761}}.

\bibitem[AV07]{AvilaVianaSimplicity}
Artur Avila and Marcelo Viana.
\newblock {Simplicity of {Lyapunov} spectra: proof of the Zorich-Kontsevich conjecture}.
\newblock {\em Acta Mathematica}, 198(1):1 -- 56, 2007.
\newblock \href {https://doi.org/10.1007/s11511-007-0012-1} {\path{doi:10.1007/s11511-007-0012-1}}.

\bibitem[BCM20]{Baake}
Michael Baake, Michael Coons, and Neil Ma{\~n}ibo.
\newblock Binary constant-length substitutions and {Mahler} measures of {Borwein} polynomials.
\newblock In {\em From analysis to visualization. A celebration of the life and legacy of Jonathan M. Borwein, Callaghan, Australia, September 25--29, 2017.}, pages 303--322. Cham: Springer, 2020.
\newblock \href {https://doi.org/10.1007/978-3-030-36568-4_20} {\path{doi:10.1007/978-3-030-36568-4_20}}.

\bibitem[BGM18]{Baake_binary_inflation_rules}
Michael {Baake}, Uwe {Grimm}, and Neil {Ma{\~n}ibo}.
\newblock {Spectral analysis of a family of binary inflation rules}.
\newblock {\em Letters in Mathematical Physics}, 108(8):1783--1805, August 2018.
\newblock \href {https://arxiv.org/abs/1709.09083} {\path{arXiv:1709.09083}}, \href {https://doi.org/10.1007/s11005-018-1045-4} {\path{doi:10.1007/s11005-018-1045-4}}.

\bibitem[BGM19]{Baake2}
Michael Baake, Franz G{\"a}hler, and Neil Ma{\~n}ibo.
\newblock Renormalisation of pair correlation measures for primitive inflation rules and absence of absolutely continuous diffraction.
\newblock {\em Commun. Math. Phys.}, 370(2):591--635, 2019.
\newblock \href {https://doi.org/10.1007/s00220-019-03500-w} {\path{doi:10.1007/s00220-019-03500-w}}.

\bibitem[BMMS24]{Bufetov-Marshal-Solomyak_Quant.Weakmixing}
Alexander~I. Bufetov, Juan Marshall-Maldonado, and Boris Solomyak.
\newblock Local spectral estimates and quantitative weak mixing for substitution $\mathbb{Z}$-actions, 2024.
\newblock URL: \url{https://arxiv.org/abs/2403.12657}, \href {https://arxiv.org/abs/2403.12657} {\path{arXiv:2403.12657}}.

\bibitem[BS18]{BufetovSolomyakHoldergenustwo}
Alexander~I. Bufetov and Boris Solomyak.
\newblock The {H}\"{o}lder property for the spectrum of translation flows in genus two.
\newblock {\em Israel J. Math.}, 223(1):205--259, 2018.
\newblock \href {https://doi.org/10.1007/s11856-017-1614-8} {\path{doi:10.1007/s11856-017-1614-8}}.

\bibitem[BS20]{Bufetov-Solomyak-spectralcocycle2020}
Alexander~I. Bufetov and Boris Solomyak.
\newblock A spectral cocycle for substitution systems and translation flows.
\newblock {\em J. Anal. Math.}, 141(1):165--205, 2020.
\newblock \href {https://doi.org/10.1007/s11854-020-0127-2} {\path{doi:10.1007/s11854-020-0127-2}}.

\bibitem[BS21]{BufetovSolomyak-HolderRegularity}
Alexander~I. Bufetov and Boris Solomyak.
\newblock H{\"o}lder regularity for the spectrum of translation flows.
\newblock {\em J. {\'E}c. Polytech., Math.}, 8:279--310, 2021.
\newblock \href {https://doi.org/10.5802/jep.146} {\path{doi:10.5802/jep.146}}.

\bibitem[BS22]{Bufetov-Solomyak-MathZ}
Alexander~I. Bufetov and Boris Solomyak.
\newblock On substitution automorphisms with pure singular spectrum.
\newblock {\em Math. Z.}, 301(2):1315--1331, 2022.
\newblock \href {https://doi.org/10.1007/s00209-021-02941-1} {\path{doi:10.1007/s00209-021-02941-1}}.

\bibitem[BS23]{BufetovSolomyak2023}
Alexander~I. Bufetov and Boris Solomyak.
\newblock Self-similarity and spectral theory: on the spectrum of substitutions.
\newblock {\em St. Petersbg. Math. J.}, 34(3):313--346, 2023.
\newblock \href {https://doi.org/10.1090/spmj/1756} {\path{doi:10.1090/spmj/1756}}.

\bibitem[BSU06]{Bufetov-Sinai-Ulcigrai}
Alexander Bufetov, Yakov~G. Sinai, and Corinna Ulcigrai.
\newblock A condition for continuous spectrum of an interval exchange transformation.
\newblock In {\em Representation theory, dynamical systems, and asymptotic combinatorics. Based on the conference, St. Petersburg, Russia, June 8--13, 2004 on the occasion of the 70th birthday of Anatoly Vershik}, pages 23--35. Providence, RI: American Mathematical Society (AMS), 2006.

\bibitem[CE19]{Chaika-Eskin-Mobius}
Jon Chaika and Alex Eskin.
\newblock M{\"o}bius disjointness for interval exchange transformations on three intervals.
\newblock {\em J. Mod. Dyn.}, 14:55--86, 2019.
\newblock \href {https://doi.org/10.3934/jmd.2019003} {\path{doi:10.3934/jmd.2019003}}.

\bibitem[CE21]{ChaikaEskin-self-joining}
Jon Chaika and Alex Eskin.
\newblock Self-joinings for 3-{IETs}.
\newblock {\em J. Eur. Math. Soc. (JEMS)}, 23(8):2707--2731, 2021.
\newblock \href {https://doi.org/10.4171/JEMS/1069} {\path{doi:10.4171/JEMS/1069}}.

\bibitem[CH17]{Chaika-Hubert}
Jon Chaika and Pascal Hubert.
\newblock Circle averages and disjointness in typical translation surfaces on every {Teichm{\"u}ller} disc.
\newblock {\em Bull. Lond. Math. Soc.}, 49(5):755--769, 2017.
\newblock \href {https://doi.org/10.1112/blms.12065} {\path{doi:10.1112/blms.12065}}.

\bibitem[Cha12]{Chaika-Annals}
Jon Chaika.
\newblock Every ergodic transformation is disjoint from almost every interval exchange transformation.
\newblock {\em Ann. Math. (2)}, 175(1):237--253, 2012.
\newblock \href {https://doi.org/10.4007/annals.2012.175.1.6} {\path{doi:10.4007/annals.2012.175.1.6}}.

\bibitem[For02]{ForniDeviation}
Giovanni Forni.
\newblock Deviation of ergodic averages for area-preserving flows on surfaces of higher genus.
\newblock {\em Ann. of Math. (2)}, 155(1):1--103, 2002.
\newblock \href {https://doi.org/10.2307/3062150} {\path{doi:10.2307/3062150}}.

\bibitem[{For}22]{ForniTwisted}
Giovanni {Forni}.
\newblock {Twisted translation flows and effective weak mixing}.
\newblock {\em J. Eur. Math. Soc.}, 24(12):4225--4276, 2022.
\newblock \href {https://doi.org/10.4171/JEMS/1186} {\path{doi:10.4171/JEMS/1186}}.

\bibitem[For24]{Forni2024-survey}
Giovanni Forni.
\newblock {\em Effective Unique Ergodicity and Weak Mixing of Translation Flows}, pages 161--221.
\newblock Springer Nature Switzerland, Cham, 2024.
\newblock \href {https://doi.org/10.1007/978-3-031-62014-0_4} {\path{doi:10.1007/978-3-031-62014-0_4}}.

\bibitem[Kat80]{Katoknomixing}
Anatole Katok.
\newblock Interval exchange transformations and some special flows are not mixing.
\newblock {\em Israel J. Math.}, 35(4):301--310, 1980.
\newblock \href {https://doi.org/10.1007/BF02760655} {\path{doi:10.1007/BF02760655}}.

\bibitem[Key88]{key1988lyapunov}
Eric~S Key.
\newblock Lyapunov exponents for matrices with invariant subspaces.
\newblock {\em The Annals of Probability}, pages 1721--1728, 1988.

\bibitem[KLR21]{Kanigowski-Lemanczyk-Radziwill}
Adam Kanigowski, Mariusz Lema{\'n}czyk, and Maksym Radziwi{\l}{\l}.
\newblock Rigidity in dynamics and {M{\"o}bius} disjointness.
\newblock {\em Fundam. Math.}, 255(3):309--336, 2021.
\newblock \href {https://doi.org/10.4064/fm931-11-2020} {\path{doi:10.4064/fm931-11-2020}}.

\bibitem[KS67]{KatokStepinWeak}
Anatole~B. Katok and Anatoly~M. Stepin.
\newblock Approximations in ergodic theory.
\newblock {\em Uspehi Mat. Nauk}, 22(5 (137)):81--106, 1967.
\newblock URL: \url{https://doi.org/10.1070/RM1967v022n05ABEH001227}.

\bibitem[Mas82]{MasurIETMeasuredFoliations}
Howard Masur.
\newblock Interval exchange transformations and measured foliations.
\newblock {\em Ann. Math. (2)}, 115:169--200, 1982.
\newblock \href {https://doi.org/10.2307/1971341} {\path{doi:10.2307/1971341}}.

\bibitem[MM24]{Marshall-Maldonado}
Juan Marshall-Maldonado.
\newblock Lyapunov exponents of the spectral cocycle for topological factors of bijective substitutions on two letters.
\newblock {\em Discrete Contin. Dyn. Syst.}, 44(7):2068--2092, 2024.
\newblock \href {https://doi.org/10.3934/dcds.2024019} {\path{doi:10.3934/dcds.2024019}}.

\bibitem[MMY05]{MMYRothtype}
Stefano Marmi, Pierre Moussa, and Jean-Christophe Yoccoz.
\newblock The cohomological equation for roth-type interval exchange maps.
\newblock {\em Journal of the American Mathematical Society}, 18(4):823--872, 2005.
\newblock \href {https://doi.org/10.1090/S0894-0347-05-00490-X} {\path{doi:10.1090/S0894-0347-05-00490-X}}.

\bibitem[MS21]{Aroundunitcircle}
James McKee and Chris Smyth.
\newblock {\em Around the unit circle. {Mahler} measure, integer matrices and roots of unity}.
\newblock Universitext. Cham: Springer, 2021.
\newblock \href {https://doi.org/10.1007/978-3-030-80031-4} {\path{doi:10.1007/978-3-030-80031-4}}.

\bibitem[NR97]{NogueiraRudolphTopweak}
Arnaldo Nogueira and Dan Rudolph.
\newblock Topological weak-mixing of interval exchange maps.
\newblock {\em Ergodic Theory Dynam. Systems}, 17(5):1183--1209, 1997.
\newblock \href {https://doi.org/10.1017/S0143385797086276} {\path{doi:10.1017/S0143385797086276}}.

\bibitem[RS23]{Rajabzadeh-Safaee}
Hesam Rajabzadeh and Pedram Safaee.
\newblock Nondegeneracy of the spectrum of the twisted cocycle for interval exchange transformations, 2023.
\newblock URL: \url{https://arxiv.org/abs/2309.05175}, \href {https://arxiv.org/abs/2309.05175} {\path{arXiv:2309.05175}}.

\bibitem[Sol]{solomyak_twisted_cocycle}
Boris Solomyak.
\newblock On the lyapunov spectrum of the twisted cocycle for substitutions.
\newblock In preparation.

\bibitem[Sol24]{Solomyak_singularity_S_adic}
Boris Solomyak.
\newblock A note on spectral properties of random $s$-adic systems, 2024.
\newblock URL: \url{https://arxiv.org/abs/2403.08884}, \href {https://arxiv.org/abs/2403.08884} {\path{arXiv:2403.08884}}.

\bibitem[ST18]{Skripchenko-Troubetzkoy}
Alexandra Skripchenko and Serge Troubetzkoy.
\newblock On the {Hausdorff} dimension of minimal interval exchange transformations with flips.
\newblock {\em J. Lond. Math. Soc., II. Ser.}, 97(2):149--169, 2018.
\newblock \href {https://doi.org/10.1112/jlms.12099} {\path{doi:10.1112/jlms.12099}}.

\bibitem[SU05]{Sinai-Ulcigrai}
YA.~G. Sinai and C.~Ulcigrai.
\newblock Weak mixing in interval exchange transformations of periodic type.
\newblock {\em Lett. Math. Phys.}, 74(2):111--133, 2005.
\newblock \href {https://doi.org/10.1007/s11005-005-0011-0} {\path{doi:10.1007/s11005-005-0011-0}}.

\bibitem[Ulc09]{Ulcigrai-JMD}
Corinna Ulcigrai.
\newblock Weak mixing for logarithmic flows over interval exchange transformations.
\newblock {\em J. Mod. Dyn.}, 3(1):35--49, 2009.
\newblock \href {https://doi.org/10.3934/jmd.2009.3.35} {\path{doi:10.3934/jmd.2009.3.35}}.

\bibitem[Ulc11]{Ulcigrai-Annals}
Corinna Ulcigrai.
\newblock Absence of mixing in area-preserving flows on surfaces.
\newblock {\em Ann. Math. (2)}, 173(3):1743--1778, 2011.
\newblock \href {https://doi.org/10.4007/annals.2011.173.3.10} {\path{doi:10.4007/annals.2011.173.3.10}}.

\bibitem[Vee82]{VeechGaussmeasure}
William~A. Veech.
\newblock Gauss measures for transformations on the space of interval exchange maps.
\newblock {\em Ann. Math. (2)}, 115:201--242, 1982.
\newblock \href {https://doi.org/10.2307/1971391} {\path{doi:10.2307/1971391}}.

\bibitem[Vee84]{VeechMetricTheory}
William~A. Veech.
\newblock The metric theory of interval exchange transformations i. generic spectral properties.
\newblock {\em American Journal of Mathematics}, 106(6):1331--1359, 1984.
\newblock URL: \url{http://www.jstor.org/stable/2374396}.

\bibitem[Via08]{viana2008dynamics}
Marcelo Viana.
\newblock Dynamics of interval exchange transformations and teichm{\"u}ller flows.
\newblock {\em available at https://www.mat.univie.ac.at/~bruin/ietf.pdf}, 2008.

\bibitem[Zor96]{ZorichGaussmeasure}
Anton Zorich.
\newblock Finite gauss measure on the space of interval exchange transformations. {Lyapunov} exponents.
\newblock {\em Annales de l'Institut Fourier}, 46(2):325--370, 1996.
\newblock \href {https://doi.org/10.5802/aif.1517} {\path{doi:10.5802/aif.1517}}.

\bibitem[Zor97]{Zorich97deviation}
Anton Zorich.
\newblock Deviation for interval exchange transformations.
\newblock {\em Ergodic Theory and Dynamical Systems}, 17(6):1477--1499, 1997.
\newblock \href {https://doi.org/10.1017/S0143385797086215} {\path{doi:10.1017/S0143385797086215}}.

\end{thebibliography}
\end{document}